\documentclass[12pt]{amsart}

\usepackage{amsthm, amsmath, amssymb, url, verbatim, xcolor, amscd,diagrams,hyperref}
\usepackage[margin=1.25in]{geometry}
{\providecommand{\noopsort}[1]{} %year = "unpublished manuscript\setbox0=\hbox{2003}"

\frenchspacing

\theoremstyle{plain}
\newtheorem{theorem}{Theorem}[section]%\newtheorem*{nonumbertheorem}{Theorem}
\newtheorem{corollary}[theorem]{Corollary}
\newtheorem{lemma}[theorem]{Lemma}
\newtheorem{proposition}[theorem]{Proposition}

\newtheorem{convention}[theorem]{Convention}
\theoremstyle{definition}

\theoremstyle{remark}
\newtheorem{remark}[theorem]{Remark}

\numberwithin{equation}{section}

\newcommand{\Z}{\mathbb{Z}}\newcommand{\Q}{\mathbb{Q}}
\newcommand{\C}{\mathbb{C}}

\DeclareMathOperator{\SO}{SO}
\DeclareMathOperator{\SU}{SU}

\DeclareMathOperator{\diag}{diag}

\newcommand{\rk}{\operatorname{rank}}

\title[Cohomogeneity one manifolds with singly generated cohomology II]{Cohomogeneity one manifolds with singly generated rational cohomology II}
\author{Jason DeVito}
\date{}

\setcounter{tocdepth}{1}
\begin{document}

%\tableofcontents

\begin{abstract} We classify cohomogeneity one actions on smooth, simply connected, closed manifolds with the rational cohomology of a sphere.  In particular, we show that such a manifold is diffeomorphic to a sphere, a Brieskorn variety, the Wu manifold $SU(3)/SO(3)$, one of two homogeneous spaces of the form $\mathbf{G}_2/SU(2)$, or a member of a particular four-parameter family of $7$-manifolds.
\end{abstract}
\bigskip

\maketitle

%\tableofcontents

\section{Introduction}

Cohomogeneity one manifolds, that is, manifolds admitting an action by a Lie group with a one-dimensional orbit space, are of fundamental importance in Riemannian geometry, especially in the study of non-negative and positive curvature.  For example, Grove and Ziller \cite{GZ} have shown that all closed cohomogeneity one manifolds admit Riemannian metrics of non-negative Ricci curvature, and that such a manifold admits a metric with positive Ricci curvature if and only if it has finite fundamental group.  In addition, such closed cohomogeneity one manifolds always admit metrics of almost non-negative sectional curvature \cite{ST}.   Cohomogeneity one manifolds are also crucial in the construction of the most recently discovered Riemannian manifold with positive sectional curvatre \cite{Dearricott11, GroveVerdianiZiller11} and have also been used to construct metrics of non-negative sectional curvature on many of the exotic $7$-spheres \cite{GroveZiller00}.  A generalization of this construction has recently been used \cite{GoetteKerinShankar20} to equip all exotic $7$-spheres with non-negatively curved Riemannian metrics.

Manifolds of non-negative sectional curvature are expected to have a relatively simple topology.  For example, a theorem of Gromov \cite{Gromov81} gives a universal bound (depending only on the dimension) on the total sum of Betti numbers of a non-negatively curved manifold and the Bott conjecture predicts that a non-negatively curved manifold should be rationally elliptic, that is, that the rational homotopy groups $\pi_k(M)\otimes\mathbb{Q}$ are trivial for $k \geq 2\dim M$  .  As such, in the search for new candidates for interesting curvature properties, one typically looks at manifolds with simpler topologies.

From the perspective of rational homotopy theory, simply connected manifolds whose rational cohomology ring is generated by a single element are the simplest - these are the spaces with precisely one non-trivial rational homotopy group in odd degree.  For homogeneous space, and more generally biquotients (the other main sources of non-negatively curved manifolds), Kapovitch and Ziller \cite{KapovitchZiller04} classified all the simply connected examples whose rational cohomology ring is singly generated.  For cohomogeneity one manifolds, the even dimensional case has recently been solved by Kennard and the author \cite{DeVitoKennard20}.  We also point out that Wang \cite{Wang60} classified cohomogeneity one actions on standard spheres (with one family missing) but only provides a full classification when the dimension is even, but not $4$, or larger than $31$.  Straume\cite{Straume96,Straume97} has classified cohomogeneity one actions on homotopy spheres, and Asoh \cite{Asoh81,Asoh83} has classified cohomogeneity one actions on $\mathbb{Z}/2\mathbb{Z}$-homology spheres.  In this paper, we generalize the work of Asoh, Straume, and Wang, and complete the classification in the odd dimensional case.

\begin{theorem}\label{thm:main}
Let $M^n$ be a simply connected, closed manifold with the rational cohomology of $S^n$ and suppose $G$ is a compact Lie group acting via a cohomogeneity one action.  Then the action is orbit-equivalent to one of the following:
	\begin{enumerate}
	\item a sphere with a linear action,
	\item a Brieskorn variety $B^{2m-1}_d$ with a standard action,  where $m \geq 3$, $d \geq 2$, and where either $m$ is $even$ or $d$ is odd,
	\item the Wu manifold $W^5 = \SU(3)/\SO(3)$ with one of infinitely many action by $S^3\times S^1$,
	\item one of two $11$-manifolds of the form $\mathbf{G}_2/SU(2)_i$ for $i=1$ or $i=3$, with action by $SU(3)\times SU(2)$, or
	\item  a $7$-manifold $M_{(p_-,q_-),(p_+,q_+)}$ with $p_- q_+ \neq p_+ q_-$ and with action $S^3\times S^3$.
	\end{enumerate}
	
	Conversely, each of these listed manifolds is a simply connected closed cohomogeneity one rational sphere.
\end{theorem} 

We actually prove a slightly stronger result.  Namely, we recover the action up to so-called normal extensions and ineffective kernel.  Given a cohomogeneity one $G$-action on a manifold $M$, if $G'\subseteq G$ is a normal subgroup of $G$ for which the restriction of the $G$-action to $G'$ is still cohomogeneity one, then we say that the $G$-action is a \textit{normal extension} of the $G'$-action.

The linear actions of Case (1) are classified by Straume \cite{Straume96}; see also \cite[Theorem 6.1]{Asoh81} for the complete list.  We also note that if $M$ is even dimensional, it must fall into this case \cite{DeVitoKennard20}.

Regarding Case (2), the Brieskorn varieties $B^{2m-1}_d$ \cite{Brieskorn66} are defined as the intersection of the zero set of the complex polynomial $z_0^d + \sum_{i=1}^m z_i^2 = 0$ with a small sphere about the origin in $\mathbb{C}^{m+1}$.  Writing $\vec{z} = (z_1,..., z_m)$, up to finite cover, the group $SO(2)\times SO(m)$ acts via $(z,A)\ast (z_0, \vec{z}) = (z^2z_0, z \vec{z} A^t)$.  When $m=8$, the restriction of this action to $SO(2)\times Spin(7)$ is still cohomogeneity one, as is the action by $SO(2)\times\mathbf{G}_2$ when $m=7$.

 When $d= 1$, $B^{2m-1}_1$ is equivariantly diffeomorphic to $S^{2m-1}$ with a linear action and when $d=2$, $B^{2m-1}$ is diffeomorphic to $T^1 S^m$.  When $m$ is even $H^m(B^{2m-1}_d)\cong \mathbb{Z}/d\mathbb{Z}$ while $B^{2n-1}$ is a homotopy sphere if both $m$ and $d$ are odd.  For completeness, we note that if $m$ is odd and $d$ is even, then $B^{2m-1}_d$ has the integral cohomology groups of $S^{m-1}\times S^m$ and when $m=2$ and $d > 1$, then $B^3_d$ is not simply connected.  We compute these cohomology groups in Appendix \ref{sec:appendix}.

In Case (3),  $W^5$ is simply connected with cohomology determined by $H_2(W^5;\Z) \cong \Z_2$.  Apart from $S^2$, it is the only example which is not $2$-connected, and the only example which is not spin.

For Case (4), the index $i$ on $SU(2)$ refers to its Dynkin index.  The two $SU(2)_i$ are the two non-trivial normal subgroups of $SO(4)\subseteq \mathbf{G}_2$.  In both cases, $SU(3)\times SU(2)_{4-i}$, acts on $\mathbf{G}_2/SU(2)_i$ by $(A,B)\ast gSU(2)_i = AgB^{-1} SU(2)_i$.  The ineffective kernel of the action is $\mathbb{Z}/2\mathbb\mathbb{Z}$ generated by $(I,-I)$.  The space $\mathbf{G}_2/SU(2)_1$ is diffeomorphic to $T^1 S^6$.  Topologically, $\mathbf{G}_2/SU(2)_3$ is the most interesting: with $\mathbb{Z}/2\mathbb{Z}$ coefficients, its cohomology ring is isomorphic to that of $S^5\times S^6$, while with $\mathbb{Z}/3\mathbb{Z}$ cofficients, it is isomorphic that of $S^3\times \mathbb{H}P^2$ \cite{McclearyZiller97}.  Every other manifold $M$ appearing in Theorem \ref{thm:main} is highly connected, and, moreover, for each $m\in \Z$, $H^\ast(M;\mathbb{Z}_m)$ is isomorphic to either $H^\ast(S^{n};\mathbb{Z}_m)$ or to $H^\ast(S^{(n-1)/2}\times S^{(n + 1)/2)}; \Z_m)$.

Finally, in Case (5), the manifolds $M_{(p_-,q_-),(p_+,q_+)}$ are found in both \cite{GroveWilkingZiller08} and \cite{EscherUltman11}; these are the $P^7_C$ manifolds of \cite{Hoelscher10}.  Each $M_{(p_-,q_-),(p_+,q_+)}$, is two-connected with third homology group isomorphic to $\Z_r$ \cite[Theorem 13.1]{GroveWilkingZiller08} where $r = p_-^2q_+^2 - p_+^2 q_-^2\neq 0$.  The $P_k$ candidates for positive sectional curvature are the subfamily consisting of the manifolds $M_{(1,1),(2k-1, 2k+1)}$.

As mentioned before, the case of rational spheres of even dimensions has already been solved \cite{DeVitoKennard20}.  This work differs from the even dimensional case in two ways.  First, the even dimensional examples always have positive Euler characteristic, which forces the existence of a singular orbit with odd codimension.  Looking at Theorem \ref{thm:GH}, this assumption leaves a limited number of possibility for the homotopy fiber of the inclusion of a principal orbit into the manifold.  This also forces at least one singular isotropy group to have full rank, with already significantly reduces the possibilities.  Second, Frank \cite{Frank13} has already classified the so-called primitive cohomogeneity one manifolds of positive Euler characteristic.  While we do show that a cohomogeneity one rational sphere must be primitive (Theorem \ref{thm:primitive}), this result is less useful in the odd dimensional case since there is no known classification of primitive cohomogeneity one manifolds when the Euler characteristic is zero.  The observation that these actions much be primitive does, however, turn out to be frequently used.

We conclude this section with an outline of the proof. Recall also that cohomogeneity one manifolds whose orbit space is an interval can be classified in terms of their so-called group diagrams \cite{Mostert57}.  A group diagram is a collection of four compact Lie groups $H, K^\pm, $ and $G$ satisfying two properties:  that $H\subseteq K^\pm \subseteq G$ and that both $K^\pm/H$ are diffeomorphic to sphere $S^{\ell_\pm}$, possibly of different dimensions, with both $\ell_\pm\geq 1$.

In dimension at most $7$, we use Hoelscher's \cite{Hoelscher10} classification, together with cohomology ring calculations of Escher and Ultman \cite{EscherUltman11} to check which examples are rational spheres.  In higher dimensions, we work as follows.  Let $\mathcal{F}$ denote the homotopy fiber of the inclusion of a principal orbit $G/H$ into $M$.  According to Grove and Halperin \cite{GroveHalperin87}, the rational homotopy type of $\mathcal{F}$ is determined up to finitely many possibilities by the orientability of the singular orbits $G/K^\pm$ and their codimensions, naturally falling into $6$ possibilities, see Theorem \ref{thm:GH}.  It turns out two of these possibilities can only occur in dimension at most $7$.

The remaining four possibilities are handled one at a time.  The first such family occurs when precisely one singular orbit is non-orientable.  In this case, we use the constrained topology of the orbits to compute the diffeomorphism types of the singular and principal orbits.  Then a simple Mayer-Vietoris argument shows that this case can only arise for actions on homotopy spheres, whose classification is already complete \cite{Straume96}.  This is done in Section \ref{sec:nonorient}.

In Section \ref{sec:exceptional}, the second and third cases are dealt with.  While the second case is trivial, the third requires significantly more work.  We first use the constrained topology to prove the groups $H,K^\pm,G$ must fall into precisely $5$ possibilities.  It turns out that each of these possibilities is encountered in \cite{Wang60}; we adapt his arguments to handle this case.

We complete the proof of Theorem \ref{thm:main}, in Section \ref{sec:last}, we handle the last case, which natural splits into two subcases depending on whether the two integers $\ell_\pm$ have same parity or not.  When they have the same parity, we again follow an approach of Wang to give a short proof that $M$ must be diffeomorphic to a standard sphere with linear action.  Lastly, we turn to the subcase where $\ell_+$ and $\ell_-$ have different parities.  This turns out to be the most interesting and most difficult case, making up roughly half of this article.  It is within this case that the Brieskorn examples in dimension $4n-1$, the actions on the homogeneous quotients of $\mathbf{G}_2$, and some some linear actions on spheres arise.   The proof in this case begins with a computation of the rational topology of the singular and principal orbits.  This knowledge, together with known classification results on the topology of Lie groups and homogeneous spaces, allows us to narrow down the form of the Lie group $G$ and the conjugacy classes of the various isotropy groups to a point where we can fully classify them.  We are then able use Asoh's classification \cite{Asoh81} to show that this information, in most cases, determines the cohomogeneity one manifold up to equivariant diffeomorphism.   For the remaining cases, we classify the group diagrams and identify them with known cohomogeneity one actions.

Finally, in Appendix \ref{sec:appendix}, we compute the cohomology groups of Brieskorn varieties.

\textbf{Acknowledgements:}  We would like to thank Lee Kennard and Krishnan Shankar for many helpful conversations.  This research was supported by NSF DMS-2105556.

\section{Preliminaries}

\subsection{Cohomogeneity one manifolds}  A smooth manifold $M$ with a smooth action of a Lie group $G$ is said to be a cohomogeneity one $G$-manifold if the orbit space $M/G$ is one-dimensional.  This is equivalent to requiring the principal orbit to have codimension $1$.  We will always restrict to the case where the manifold $M$ is closed (compact and boundaryless) and simply connected, and where $G$ is compact and connected.  A fundamental structure result of Mostert \cite{Mostert57} asserts the following.

\begin{theorem} Suppose $M$ is a simply connected closed cohomogeneity one manifold under the action of a compact Lie group $G$.  Then $M/G$ is homeomorphic to $[-1,1]$ and 

\begin{itemize}\item  points in $(-1,1)$ correspond to principal orbits $G/H$,

\item  the boundary points $\pm 1\in [-1,1]$ correspond to singular orbits $G/K^\pm$,

\item  both $K^\pm/H\cong S^{\ell_\pm}$ are diffeomorphic to spheres of positive dimension $\ell_\pm$, and

\item $M$ is equivariantly diffeomorphic to the union of two disk bundles $D(G/K^\pm) = G\times_{K^\pm} D^{\ell_\pm + 1}$ over $G/K^\pm$ glued together along their common boundary $G/H\cong G\times_{K^\pm} S^{\ell_\pm}$ via the identity map.

\item  Conversely, a group diagram $H\subseteq K^\pm \subseteq G$ with $K^\pm/H\cong S^{\ell_\pm}$ with both $\ell_\pm \geq 1$ uniquely determines a cohomogeneity manifold $M = D(G/K^+)\cup_{G/H} D(G/K^-)$.

\end{itemize}

\end{theorem}

We say $M$ is a \textit{double disk bundle} as it can be expressed as a union of two disk bundles glued along their common boundary.  This decomposition of $M$ will be frequently used in combination with the Mayer-Vietoris sequence in cohomology to relate the relatively simple topology of $M$ with that of the principal and singular leaves.  To that end note that $D(G/K^\pm)$ is homotopy equivalent to $G/K^\pm$, with a deformation retract given by contracting the disk factor to a point.  The composition $G/H \cong G\times_{K^\pm} K^\pm/H \subseteq D(G/K^\pm) \rightarrow G/K^\pm$ is simply the canonical projection map $G/H\rightarrow G/K^\pm$.  In particular, if whenever we apply the Mayer-Vietoris sequence to the double disk bundle decomposition of $M$, we will always identify the disk bundles $D(G/K^\pm)$ with $G/K^\pm$ and the inclusion $G/H\rightarrow D(G/K^\pm)$ with the projection $G/H\rightarrow G/K^\pm$.

We allow the $G$-action to have non-trivial ineffective kernel, so we can always pass to a cover of $G$.  For the duration of the paper, we will adopt the following convention regarding cohomogeneity one actions of a compact Lie group $G$ on a manifold $M$:

\begin{convention}\label{convention}:  \begin{itemize}\item  $G\cong G_0\times T^k$ is isomorphic to the product of a simply connected Lie group $G_0$.

 \item  $H\subseteq G$ will always denote the principal isotropy group.

\item  $K^\pm\subseteq G$ will denote the two singular isotropy groups.

\item  $\ell_\pm$ will denote the dimension of the fiber spheres $K^\pm/H\cong S^{\ell_\pm}$.

\item  No proper normal subgroup of $G$ acts with the same orbits as the $G$-action.

\item The phrase ``$M$ is a rational sphere" is an abbreviation for ``$M^n$ is a closed smooth simply connected manifold whose rational cohomology is isomorphic to $H^\ast(S^n;\Q)$''. 

\end{itemize}
\end{convention}

We remark that in terms of the group diagram, the condition that no normal subgroup of $G$ acts with the same orbits is equivalent to requiring that the projection of $H$ to any simple factor of $G$ is not surjective.  This also implies that the ineffective kernel of the action is at most finite.

We will occasionally need to deal with disconnected Lie groups.  In this case, we use the notation $L^0$ to denote the identity component of a Lie group $L$.

The pair of integers $\ell_\pm$ will play an important role in what follows, so we record some preliminary information regarding them.  See, e.g., \cite[Lemma 1.6]{GroveWilkingZiller08} for a proof.

\begin{proposition}\label{prop:ellinfo}  Suppose $M$ is a closed simply connected $G$-manifold with $G$ a compact Lie group acting via a cohomogeneity one action.  Then 

\begin{itemize}\item  If $G/K^\pm$ is not simply connected, then $\ell_\mp = 1$, and

\item  In particular, if both $\ell_\pm \geq 2$, the all three orbits $G/K^\pm$ and $G/H$ are simply connected, and $H$ and $K^\pm$ are connected.

\item  In addition, if $\ell_+ = 1$ and $\ell_- > 1$, then $H = H^0\cdot \mathbb{Z}_k$, $K^+ = H^0\cdot S^1 = H\cdot S^1$, and $K^- = (K^-)^0\cdot \mathbb{Z}_k$.

\end{itemize}

\end{proposition}\label{prop:equiv}

If one wants to classify actions up to $G$-equivariant diffeomorphism, it is important to be able to recognize when two diagrams correspond to $G$-equivariantly diffeomorphic manifolds.  This problem is solved in \cite{GroveWilkingZiller08}, see also \cite[Proposition 1.3]{Hoelscher10}.

\begin{proposition}  Suppose $H\subseteq K^\pm \subseteq G$ is a group diagram.  Then any combination of the following modifications to the diagram gives rise to an equivalent $G$-manifold.

\begin{itemize}\item  Switching $K^+$ and $K^-$

\item  Conjugating each group in the diagram by the same element of $G$

\item  Replacing $K^-$ with $aK^- a^{-1}$ for any $a\in N(H)_0$, the identity component of the normalizer of $H$ in $G$.

\end{itemize}

Conversely, two diagrams describing $G$-equivariantly diffeomorphic manifolds must be related by some combination of these three operations.

\end{proposition}

\subsection{Important groups and subgroups}

In this subsection, we record several results which will be used repeatedly.  The first, due to Borel \cite{Borel53}, is the classification of transitive actions of compact Lie groups on spheres.

\begin{proposition}  Suppose $G$ is a compact Lie group acting effectively and transitively on a sphere with isotropy group $H$.  Then $G$ and $H$ are given in Table \ref{table:transsphere}.  Conversely, for each pair $(G,H)$ in the table, $G$ acts effectively and transitively on a sphere with stabilizer $H$.

\end{proposition}

\begin{table}

\begin{center}

\begin{tabular}{c|c|c}

$G$ & $H$& dimension\\

\hline

$ SO(m)$ & $SO(m-1)$ & $m-1$\\

$SU(m)$ & $SU(m-1)$ & $2m-1$ \\

$SU(m)\times S^1$ & $SU(m-1)\times S^1$ & $2m-1$\\

$Sp(m)$ & $Sp(m-1)$ & $4m-1$\\

$Sp(m)\times S^1$ & $Sp(m-1)\times S^1$ & $4m-1$\\

$Sp(m)\times Sp(1)$ & $Sp(m-1)\times Sp(1)$ & $4m-1$\\

$G_2$ & $SU(3)$ &$6$  \\

$Spin(7)$ & $G_2$ & $7$\\

$Spin(9)$ & $Spin(7)$ & $15$\\

\end{tabular}
\caption{All effective transitive actions on spheres}
\label{table:transsphere}
\end{center}
\end{table}

In all cases but the first in Table \ref{table:transsphere}, the fact that $G/H$ is a sphere with $H$ acting effectively characterizes the embedding of $H$ into $G$ up to conjugacy.  If the first entry $(SO(m),SO(m-1))$ is lifted to the universal cover $Spin(m)$, then the embedding of $Spin(m-1)$ into $Spin(m)$ is unique up to conjugacy except when $m=4$ or $m=8$.  When $m=4$, $(Spin(4),Spin(3)) = (SU(2)\times SU(2), SU(2))$, and there three such embeddings up to conjugacy, given by mapping $SU(2)$ in either diagonally, or into each factor.  When $m=8$, $Spin(7)$ embeds into $Spin(8)$ in three non-conjugate ways.  The action of the outer automorphism group of $Spin(8)$ permutes them.

Due to the exceptional isomorphisms $SU(2)\cong Spin(3)$, $SU(2)^2 \cong Spin(4)$, $Sp(2)\cong Spin(5)$, and $SU(4)\cong Spin(6)$, we also note the following.

\begin{proposition}\label{prop:mix}  If $G$ is a simple Lie group acting transitively on two different spheres of dimensions $m<n$, then $$(G,m,n)\in\{ (Spin(9), 8, 15), (Spin(7), 6,7), (SU(2), 2,3), (Sp(2), 4,7), (SU(4), 5,7)\}.$$  If $H$ is a simple Lie group appearing as the isotropy group of a transitive action of a simple group on spheres of dimensions $m<n$, then $$(H,m,n)\in \{(Spin(7), 7,15), (SU(2), 3,5), (SU(3), 6,7), (Sp(2), 5,11), (SU(4), 6, 9)\}.$$

\end{proposition}

Finally, we record Table \ref{table:corank2} which will be used in two separate arguments.  Table \ref{table:corank2} consists of all pairs $(G,L)$ where $G$ is a simple simply connected Lie group, $L$ is corank $2$ in $G$, $L$ is simple, and $\dim \pi_{odd}(G/L)\otimes \Q = 2$.  The third and fourth columns indicate that $\pi_{\ell_-}(G/L)\otimes \Q \neq 0$ and $\pi_{\ell_- + \ell_+}(G/L)\otimes \Q\neq 0$, and the $\ell_+$ column is the difference of the previous two columns.   In the last column, restrictions on $n$ are given, and the word``multiple" indicates that there is more than one conjugacy class of embedding of $L$ into $G$.  The first two columns are found in \cite[Table 11]{Onishchik97}, while the computations of the remaining three columns are elementary using the long exact sequence of rational homotopy  groups associated to the bundle $L\rightarrow G\rightarrow G/L$, together with, e.g., \cite[Table 4.2]{Totaro02}, which lists all the non-trivial rational homotopy groups of simple Lie groups.  If the embedding of $L$ into $G$ is not specified, the obvious embedding is intended.

\begin{table}

\begin{tabular}{c|c|c|c|c|c}

$G$ & $L$ & $\ell_-$ & $\ell_- + \ell_+$ & $\ell_+$ & Notes\\

\hline

$SU(m)$ & $SU(m-2)$ & $2m-3$ & $2m-1$ & $2$ & $m\geq 3$, multiple when $m=4$\\

$SU(6)$ & $SO(6)$ & $9$ & $11$ & $2$ & \\

$SU(6)$ & $Sp(3)$ & $5$ & $9$ & $4$ &\\

$SU(5)$ & $Sp(2)$ & $5$ & $9$ & $4$ & multiple\\

%$SU(5)$ & $SO(5)$ & $5$ & $9$ & $4$ &\\

%$SU(4)$ & $SU(2)$ & $5$ & $7$ & $2$ & Irreducible $4$-dim rep\\

$Spin(2m+1)$ & $Spin(2m-3)$ & $4m-5$ & $4m-1$ & $4$ &$m\geq 4$, multiple when $m=3$\\

$Spin(9)$ & $Sp(2)$ & $11$ & $15$ & $4$ \\

$Spin(9)$ & $\mathbf{G}_2$ & $7$ & $15$ & $8$ \\

%$Spin(7)$ & $SU(2)$ & $7$ & $11$ & $4$ & multiple\\

$Sp(m)$ & $Sp(m-2)$ & $4m-5$ & $4m-1$ & $4$ & $m\geq 2$, multiple when $m=3$\\

%$Sp(3)$ & $Sp(1)$ & $7$ & $11$& $4$ & Several embeddings \\

$Spin(2m)$ & $Spin(2m-3)$ & $2m-1$ & $4m-5$ & $2m-4$ & $m\geq 4$\\

$Spin(8)$ & $\mathbf{G}_2$ & $7$ & $7$ & $0$\\

$\mathbf{E}_6$ & $\mathbf{F}_4$ & $9$ & $17$ & $8$\\

$\mathbf{F}_4$ & $\mathbf{G}_2$ & $15$ & $23$ & $8$\\

$\mathbf{G}_2$ & $\{e\}$ & $3$ & $11$ & $8$ & 

\end{tabular}

\caption{Pairs $(G,L)$ with $\dim \pi_{odd}(G/L)\otimes \Q = 2$ and $L$ corank $2$, simple}
\label{table:corank2}
\end{table}

\subsection{The Grove-Halperin model of a cohomogeneity one manifold}

In \cite{GroveHalperin87}, Grove and Halperin study double mapping cylinders from the perspective of rational homotopy theory.  Using the notation $\simeq_\Q$ to denote rational homotopy equivalence, their main result, adapted to our situation, can be summarized as follows.

\begin{theorem}[Grove-Halperin]\label{thm:GH} Let $\mathcal F$ be the homotopy fiber of the inclusion $G/H \to M$ of the principal orbit of a cohomogeneity one manifold. Let $h$ denote the number of non-orientable singular orbits.  Then one of the following occurs.
	\begin{enumerate}
	\item $h = 2$, both $\ell_\pm = 1$, and up to finite cover $\mathcal{F}$ has the rational homotopy type of $S^3 \times S^3 \times \Omega S^7$.
	\item $h = 1$, both $\ell_\pm = 1$, and up to finite cover $\mathcal{F}$ has the rational homotopy type of $S^1 \times S^3 \times \Omega S^5$.
	\item $h = 1$, $\ell_+ = 1$, $\ell_- \geq 3$ is odd, and $\mathcal F \simeq_\Q S^1 \times S^{2 \ell_- + 1} \times \Omega S^{2 \ell_- + 3}$.
	\item $h = 0$ and $\mathcal F \simeq_\Q S^{\ell_1} \times S^{\ell_+} \times \Omega S^{\ell_- + \ell_+ + 1}$.
	\item $h = 0$, $l_+ = \ell_-$ is even, and $\mathcal F \simeq_\Q S^{\ell_-} \times \Omega S^{\ell_- + 1}$.
	\item $h=0$, and $\mathcal F$ is rationally homotopy equivalent to one of $$ SU(3)/T^2 \times \Omega S^7 ,Sp(2)/T^2\times \Omega S^9 ,\mathbf{G}_2/T^2\times \Omega S^{13}$$ with $\ell_- = \ell_+ = 2$,  to $Sp(3)/Sp(1)^3\times \Omega S^{13}$ with $\ell_- = \ell_+ = 4$, or to $\mathbf{F}_4/Spin(8)\times \Omega S^{25}$ with $\ell_- = \ell_+ = 8$.
	\end{enumerate}
\end{theorem}

Writing $\pi_\ast^\Q(M)$ as shorthand for $\pi_\ast(M)\otimes \Q$, it is clear that knowledge of the connecting homomorphism $\partial:\pi_k^\Q(M)\rightarrow \pi_{k-1}^\Q(\mathcal{F})$ will directly control the topology of $G/H$.  Note in every case that, up to a finite cover, $\mathcal F$ has the rational homotopy groups of a closed, simply connected manifold $X$ and a space of the form $\Omega S^{2k+1}$. This is obvious in all cases except possibly in Case (4) when $\ell_- + \ell_+ +1$ is even. However in this case, one uses the fact that $\Omega S^{2t} \simeq_\Q S^{2t-1} \times \Omega S^{4t-1}$.  Combining \cite[Theorem 4.11]{DeVitoGalazGarciaKerin} and \cite[Corollary 2.6]{DeVitoKennard20}, we obtain the following proposition.

\begin{proposition}\label{prop:Connecting}
For a closed, simply connected cohomogeneity one manifold $M$ with principal orbit $G/H$, and homotopy fiber $\mathcal F$ as in Theorem \ref{thm:GH}, the connecting homomorphism $\partial_{odd} : \pi_{odd}^\Q(M) \to \pi_{even}^\Q(\mathcal F)$ has rank one.

Moreover, the image of $\partial_{odd}$ is supported by the the loopspace factor of $\mathcal F$ in the following sense: decomposing a finite cover of $\mathcal F$ as $X\times \Omega S^{2k+1}$, then $\partial_{2k+1} \neq 0$.
\end{proposition}

%\begin{proof}  The first statement is \cite[Corollary 3.6]{DeVitoKennard20}.  On the other hand, from \cite{Theorem4.11]{DeVitoGalazGarciaKerin}, So we focus on the second statement.

%Note that this is immediate in Cases (1) - (3) of Theorem \ref{thm:GH} since $\dim \pi_{odd}^\Q{\mathcal{F}} = 1$.  To prove the result in the remaining cases, we follow the proof of \cite[Proposition 3.7]{DeVitoKennard20}.   Supposing $\pi_{2k}^\Q(-)$ of the loopspace factor of $\mathcal{F}$ is non-zero, we need to demonstrate that $d_0 : V_F^{2j} \to V_M^{2j+1}$ is zero for all $j\neq k$. The proof there shows that if $y \in V_F^{2k}$ satisfies $y^2 \in \im(d_F)$, then $d_0(y) = 0$. This already completes the proof in Cases (4) and (5) of Theorem \ref{thm:GH}.

%As for Case (6), from \cite{GroveHalperin87}, for each of the five possibilities for $X$, $H^\ast(X)$ is generated as an algebra by two elements $x$ and $y$ of degree $\ell_-=\ell_+\in \{2,4,8\}$, subject to the relation $x^2 + \alpha y^2 = 0$ for $\alpha\in\{1,3\}$.  Abusing notation and writing $x,y\in \Lambda X\subseteq \Lambda F\otimes \Lambda M$, we must show that $d_0(x) = d_0(y) = 0$.  In the minimal model for $X$, there must be an element $z$ of degree $4\ell_- - 1$ with $x^2 + y^2 = d_F(z)$. Computing as in the proof of \cite[Proposition 2.7]{DeVitoKennard20} again, we find expanding the equation $d^2_{G/H}(z) = 0$ yields 	\[0 = 2 x d_0(x) + 2\alpha y d_0(y)\]
%in $\Lambda F \tensor\Lambda M$. But $x$ and $y$ are linearly independent in $\Lambda F$, so this implies that $d_0(x) = d_0(y) = 0$, as required.
%\end{proof}

\subsection{Limits on the group diagrams}\label{sec:limits}

In the section, we find some general necessary conditions on the diagram $H\subseteq K^\pm \subseteq G$ in order for the resulting cohomogneity one manifold to be a simply connected rational sphere.

The main result of this subsection is that all cohomogeneity one actions on simply connected rational spheres are primitive.  Recall that a cohomogeneity one action of $G$ on $M$ is called \textit{non-primitive} if it has a group diagram $H\subseteq K^\pm \subseteq G$ for which one can find an intermediate subgroup $H\subseteq K^\pm \subseteq L \subsetneq G$, and is called \textit{primitive} if it is not non-primitive.

\begin{theorem}\label{thm:primitive}  Suppose $M$ is a closed simply connected cohomogeneity one manifold under the action of a connected compact Lie group $G$.  If the rational cohomology ring of $M$ is generated by one element, then the action is primitive.

\end{theorem}

This theorem was proven in \cite{DeVitoKennard20} under the additional assumption that $M$ had positive Euler characteristic.  Here we give a proof which handles all possibilities for the Euler characteristic simultaneously.

Let us first set up some notation.  To begin with, as usual we have the group diagram of a non-primitive action $H\subseteq K^\pm\subseteq L\subseteq G$ which gives rise to two cohomogeneity one manifolds $M_L$ and $M_G$.  The space $M_G$ is equivariantly diffeomorphic to $G\times_L M_L$, where $G$ acts by left multiplication on the first factor; we have an inclusion $j:M_L\rightarrow M_G$ given by $j(m) = [e,m]\in G\times_L M_L$.

%This construction is actually quite general: given any action of a Lie group $L$ on a space $M$ and $L\subseteq G$, then we can define a $G$-manifold $G\times_L M$ with natural inclusion $j:M\rightarrow G\times_L M$.  For $p\in M$ a regular point, i.e., with isotropy group $L_p = H$ principal, it is easy to see that $j(p)\in G\times_L M$ is also regular, with the same isotropy group.

The main technical tool for proving Theorem \ref{thm:primitive} is the following.

\begin{lemma}\label{lem:pullback}Suppose $L\subseteq G$ are compact Lie groups, and that $L$ acts smoothly on a manifold $M_L$.  Choose $p\in M_L$ with $L_p = H$ as principal isotropy group and let $M_G = G\times_L M_L$ where $g\ast[g_1,x] = [gg_1,x]$, so $H = G_{[e,p]}$.  Let $j:M_L\rightarrow M_G$ be defined by $j(x) = [e,x]$.  Let $F_L$ be the homotopy fiber of the inclusion $L/H\xrightarrow{i_L} M_L$ and likewise, let $F_G$ be the homotopy fibration of the inclusion $G/H\xrightarrow{i_G} M_G$.  Then the fibration $F_L\rightarrow L/H\xrightarrow{i_L} M_L$ is the pullback of the fibration $F_G\rightarrow G/H\xrightarrow{i_G} M_G$ under $j$.  In particular, $F_L$ is homotopy equivalent to $F_G$.

\end{lemma}

\begin{proof}

We construct a diffeomorphism from $L/H$ to $j^\ast(G/H)$ which covers the identity map on $M$.  Recall that by definition, $$j^\ast(G/H) = \{(x,gH)\in M_L\times G/H: j(x) = i_G(gH)\}.$$  The condition $j(x) = i_G(gH)$ is equivalent to $[e,x] = [g,p]\in M_G$.  It follows that $(x,gH)\in j^\ast(G/H)$ iff $g\in L$ and $x = gp$.

Now, we define $\phi:L/H\rightarrow j^\ast(G/H)$ by $\phi(gH) = (i_L(gH), gH)$; it is easy to verify that this is smooth and that is covers the identity map on $M_L$.  It is also easy to verify that $\phi^{-1}(x,gH) = gH$ is the inverse, so $\phi$ is a diffeomorphism.

\end{proof}

We need another lemma.

\begin{lemma}\label{lem:connected}  Suppose $H\subseteq K^\pm \subseteq G$ is the group diagram of a simply connected manifold $M$.  If the action is non-primitive, then there is a connected proper subgroup $L$ containing $H, K^\pm$.

\end{lemma}

\begin{proof}%In \cite[Proposition 3.1]{DeVitoKennard20}, this is proven under the assumption that at least one of $K^\pm$ is connected.  Hence, we may assume both $K^\pm$ are disconnected.  From Proposition \ref{prop:ellinfo}, this implies that both $\ell_\pm = 1$.

Since the action is non-primitive, there is a proper subgroup $L$ of $G$ containing all of $H,K^\pm$.  We claim that $L^0$, the identity component of $L$, also contains $H$ and both $K^\pm$.  We define $H^\pm = H\cap (K^\pm)^0$.  Then, according to \cite[Lemma 1.10]{Hoelscher10}, the fact that $M$ is simply connected implies $H = H^+\cdot H^-$.  Now, since $H^+\subseteq (K^+)^0\subseteq L^0$, and likewise $H^-\subseteq L^0$, it follows that $H\subseteq L^0$.

Now, let $K$ be a component of either $K^\pm$.  Because $K^\pm/H = S^{\ell_\pm}$ is connected, there is an element $h\in H\cap K\subseteq L^0\cap K$ so $K$ intersects $L^0$.  However, since $K\subseteq L$ and $K$ is connected, it must be contained in precisely one component of $L$.  Thus, $K\subseteq L^0$.

\end{proof}

We can now prove Theorem \ref{thm:primitive}.

\begin{proof}(Proof of Theorem \ref{thm:primitive}) 
By computing a Sullivan model, it is easy to see that $H^\ast(M_G;\Q)$ is singly generated and finite dimensional if and only if $\dim \pi_{odd}^\Q(M_G) = 1$ and $\dim \pi^\Q_{even}(M_G)\leq 1$.

Suppose we have a chain of subgroups $H\subseteq K^\pm \subseteq L\subsetneq G$.  From Lemma \ref{lem:connected}, we may assume $L$ is connected, which implies that $M_L$ is connected.  Then we get a fiber bundle $M_L\rightarrow M_G\rightarrow G/L$ with $G/L$ simply connected.  Because the map $\pi_{odd}^\Q(M_G)\rightarrow \pi_{odd}^\Q(G/L)$ is surjective (see, e.g., \cite[Corollary 3.5]{DeVitoKennard20}and $\dim \pi_{odd}^\Q X \geq 1$ for any rationally elliptic space $X$, it follows that $G/L$ also has singly generated rational cohomology.  If $\pi_2^\Q(G/L)\neq 0$, it follows from \cite[Theorem 32.6(i)]{FelixHalperinThomas01} that $\dim(G/L) = \dim M_G - 1$, so $M_L$ is one dimensional, and hence $M_L\cong S^1$.  In this case, the homotopy fiber of the inclusion of a principal orbit (a point) into $M_L$ is homotopy equivalent to $\mathbb{Z}$.  On the other hand, in Theorem \ref{thm:GH}, all of the homotopy fibers are connected.  Thus, we have contradicted Lemma \ref{lem:pullback}, so we must have $\pi_2^\Q(G/L) = 0$.

From the classification of biquotients with singly generated rational cohomology ring \cite{KapovitchZiller04}, it follows that $\pi_2(G/L) = 0$ so that, in particular, $M_L$ is simply connected, so Proposition \ref{prop:Connecting} can be applied.  Now, suppose $\pi_s^\Q(M_G)\neq 0$ with $s$ odd and note that $s\geq \dim M_G$.  From Lemma \ref{lem:pullback} and Proposition \ref{prop:Connecting}, it follows that $\pi_s(M_L) \neq 0$ as well.  Further, since $\pi_{odd}^\Q(M_G)$ surjects onto $\pi_{odd}^\Q(G/L)$, and $\dim\pi_{odd}^\Q(M_G) = 1$, the map $\pi_s^\Q(M_L)\rightarrow \pi_s^\Q(M_g)$ must be trivial.  In particular, we conclude that $\pi_{s+1}^\Q(G/L)\neq 0$.  Since $s+1$ is even and $s+1 > \dim M_G > \dim G/L$, we have contradicted \cite[Theorem 32.6(ii)]{FelixHalperinThomas01}.
\end{proof}

We also mention a one consequence of primitivity.  See \cite[Lemma 3.4]{GroveWilkingZiller08} for a proof.

\begin{proposition}\label{prop:finiteintersection}  Suppose $H\subseteq K^\pm \subseteq G$ is a group diagram of a simply connected manifold.  Denote by $H_\pm$ the ineffective kernel of the action of $H$ on $K^\pm/H$.  If the action is primitive, then $H_+\cap H_-$ is finite.

\end{proposition}

\section{Dimensions up to seven}\label{sec:atmost7}
Cohomogeneity one actions on closed, simply connected manifolds have been classified in dimensions at most $7$ by Parker \cite{Parker}, Neumann \cite{Neumann}, and Hoelscher \cite{Hoelscher10}.  As the case of cohomogeneity one even-dimensional rational spheres has already been completed, we focus on the case when $n\in \{3,5,7\}$ is odd.

In dimension $3$, there is only the sphere for topological reasons.  Further, the only cohomogeneity one action is, up to equivalence, the standard $T^2$ action on $S^3$.

In dimension 5, Hoelscher proved that if $M^5$ is a rational sphere, then it is diffeomorphic to $S^5$ or the Wu manifold $SU(3)/SO(3)$.  Further, when $M\cong SU(3)/SO(3)$, $G$ must always be, up to cover, isomorphic to $S^3\times S^1$.

In dimension 7, Hoelscher classified the group diagrams (see, for example, \cite[Section 1.1]{Hoelscher10}).  He shows that up to equivariant diffeomorphism, $M^7$ is 
	\begin{enumerate}
	\item a symmetric space,
	\item a product of a cohomogeneity one manifold and a homogeneous space,
	\item a Brieskorn variety $B_d^7$, or
	\item one of four primitive examples $P^7_A,\ldots,P^7_D$ or one of nine non-primitive examples $N^7_A,\ldots,N^7_I$.
	\end{enumerate}
	
We must determine which of these are rational spheres.  First, the classification of symmetric spaces implies that the only one with the rational cohomology of $S^7$ is the standard $S^7$.

As for Case (2), the factors are simply connected, closed manifolds since $M$ is, so the K\"unneth formula implies that $M$ is not a rational cohomology sphere.

For Case (3), we compute more generally the cohomology of Brieskorn varieties in Appendix \ref{sec:appendix}.  For dimension seven in particular, we have $B^7_d$ is always a rational sphere with $H^4(B^7_d)\cong \mathbb{Z}_d$.  

%In particular, the assumption that $M$ is a rational sphere implies that $M$ is a homotopy sphere, so this case falls under Straume's classification.
%we note the following (see \cite{GroveVerdianiWilkingZiller??}): $B_d^7$ is simply connected (as are all Brieskorn varieties of dimension greater than three). For small $d$, we have that $B_1^7 = S^7$ and $B_2^7 = T^1S^4$, which is 2-connected and has $H_3(B_d^7;\Z) \cong \Z_2$. For odd $d \geq 3$, we have that $B^7_d$ is a mod 2 homology sphere but not a homotopy sphere (see \cite[Example 1.2]{Asoh81} and the introduction to Asoh83, which gives a reference that $B_d^{2n-1}$ is a homotopy sphere if and only if both $d$ and $n$ are odd\footnote{Also see: https://www.uni-math.gwdg.de/tammo/preprints/space.pdf  This explains exactly when $V_{a}^{2n-1} = S^{2n+1} \cap \{z_0^{a_0} + \ldots + z_n^{a_n} = 0\} \subseteq \C^{n+1}$ is a (integral) homology sphere. Namely, $V_a$ is a homology (therefore homotopy for $n \geq 2$ or $3$) sphere if and only if one of the following holds (up to reordering the $a_i$ -- result is due to Brieskorn1966):
%	\begin{enumerate}
%	\item $\gcd(a_0,a_i) = \gcd(a_1,a_i) = 1$ for all $1 \leq i \leq n$, or
%	\item $\gcd(a_0,a_i) = 1$ for all $1 \leq i \leq n$, $\gcd(a_i,a_j) = 2$ for all $1 \leq i < j \leq n$, and $n$ is odd (i.e., $\dim V_a \equiv 1 \bmod 4$). For example, the polynomial $z_0^{2d+1} + \sum_{i=1}^{n} z_i^2$ satisfies these conditions if $n$ are odd.)
%	\end{enumerate}}
% (and $d \geq 3$? since otherwise it is always a sphere?)). For even $d \geq 4$, \Lee{I don't know...}

Finally, for Case (4), we first note that the non-primitive examples are ruled out using Theorem \ref{thm:primitive}.  The cohomology rings of the primitive examples are computed in \cite[Section 3]{EscherUltman11}. In the notation of Escher and Ultman, the four classes are called $L$, $M$, $N$, and $O$. For the $L$, $N$, and $O$, classes, $H^2 = \Z$, so these are not rational spheres. The manifolds in the $M$ family (which are the $P^7_C$ family in \cite{Hoelscher10}) are parameterized by four integers $p_\pm, q_\pm$, all congruent to $1$ mod $4$.  Setting $r = |p_-^2 q_+^2 -p_+^2 q_-^2|/8$, $M_{(p_\pm, q_\pm)}$ is a rational sphere iff $r\neq 0$.  Further when $r\neq 0$, $M_{p_\pm, q_\pm}$ is a 2-connected rational $7$-sphere with $H^4 \cong \Z_r$.

As a quick remark, note that for any integer $t$, precisely one of $\pm(2t + 1)$ is congruent to $1$ mod $4$, and similarly for $\pm(2t-1)$.  If we set $p_+ = \pm(2t+1)$, $p_- = \pm(2t-1)$, and $q_\pm = 1$ with the $\pm$ signs chosen to make both $p_\pm$ congruent to $1$ mod $4$, then it is simple to verify that $r = t$.  In other words, for any integer $t$, the torsion group $\Z/t\Z$ appears as $H^4(M_{p_\pm, q_\pm};\Z)$ for an appropriate choice of $p_\pm$ and $q_\pm$.

We summarize all this in the following corollary.

\begin{corollary}\label{cor:dim7orless}If $M$ is cohomogeneity one rational sphere with $\dim M \in \{3,5,7\}$, then up to ineffective kernel in the action, $M$ is equivariantly diffeomorphic to one of the following:
\begin{enumerate}
\item a standard sphere $S^n$ with linear action.
\item the Wu manifold $W^5 = \SU(3)/\SO(3)$ with an action by $G = S^3\times S^1$.
\item a Brieskorn variety $B^7_d$ with standard action by $G = SO(2)\times SO(4)$
\item $M_{(p_\pm, q_\pm)}$, where $p_-^2 q_+^2 \neq p_+^2 q_-^2$, with standard action by $G = S^3\times S^3$.
\end{enumerate}
Moreover, all of these manifolds admit such a cohomogeneity one action.
\end{corollary}

\section{Case (1), (2), and (3) of Theorem \ref{thm:GH}}\label{sec:nonorient}

In this section, we prove Theorem \ref{thm:main} under the assumption that at least one singular orbit is non-orientable.  We recall Convention \ref{convention} and that $n = \dim M$ is always assumed to be odd. The first step is to prove the following.

\begin{proposition}\label{thm:h=2}
Suppose $M^n$ is a rational sphere which admits a cohomogeneity one action by a compact Lie group $G$ such that both singular orbits are non-orientable, then $n = 7$ and $M$ is classified by Corollary \ref{cor:dim7orless}.
\end{proposition}

\begin{proof}
Since both singular orbits are non-orientable, this must fall into Case (1) of Theorem \ref{thm:GH}.  In particular, Proposition \ref{prop:Connecting} implies that $\pi_7^\Q(M) \neq 0$. Since $M$ is a rational sphere, the result follows.
\end{proof}

Moving on to the case where exactly one orbit is non-orientable, we have the following:

\begin{theorem}\label{thm:h=1}
Suppose $M^n$ is a rational sphere which is a cohomogeneity one manifold for a group action with precisely one non-orientable singular orbit.  Then one of the following occurs
	\begin{enumerate}
	\item  $M$ is a homotopy sphere
	\item  $M = \SU(3)/\SO(3)$
	\end{enumerate}
\end{theorem}

As mentioned above, Straume \cite{Straume96,Straume97} has already classified the cohomogeneity one actions on homotopy spheres: such an action is orbit equivalent to a linear action on a sphere or a standard action on a Brieskorn variety.  Likewise, the Wu manifold $SU(3)/SO(3)$ is five-dimensional, so the cohomogeneity one actions are classified by Corollary \ref{cor:dim7orless}.

The proof of Theorem \ref{thm:h=1} requires the rest of this section.  Since precisely one singular orbit is non-orientable, we must be in Case (2) or (3) of Theorem \ref{thm:GH}. Moreover we may assume $\dim M > 7$ by Corollary \ref{cor:dim7orless}, so by Proposition \ref{prop:Connecting}, we must be in Case (3).  Thus, for the remainder of this section, we may assume that $G/K^-$ is non-orientable and $\ell_+ = 1$, and that $\ell_-\geq 3$ is odd. In particular, $G/K^+$ is simply connected and $K^+$ is connected by Proposition \ref{prop:ellinfo}.

The proof proceeds by a sequence of propositions that compute the diffeomorphism types of the singular and principal orbits. Once this is done, the proof of Theorem \ref{thm:h=1} is an easy consequence of the Mayer-Vietoris sequence applied to the double disk bundle decomposition of $M$.

%In dimension $5$, everything is in Corey's thesis, Theorem C, which says the only rational spheres in dim $5$ which are cohom 1 are $S^5$ and $\SU(3)/SO(3)$.  The $\SU(3)/SO(3)$ case is the $P^5$ thing, and certainly has times with one non-orientable orbit.  The other relevant thing is $Q^5_B$ (the $N$s are proven to be rational $S^2\times S^3$s and the others are non-primitive or don't have both orbits circles.  Presumably, each $Q^5_B$ thing is diffeo to $S^5$, but I'm not positive.\footnote{J:It now follows from results of Martin, Fernando, and me that a double disk bundle in dimension $5$ which is rationally $S^5$ is diffeomorphic to $S^5$ or $\SU(3)/SO(3)$. No group action needed}

\begin{proposition}\label{prop:gkh=1} The bundle $S^1 \rightarrow G/H\rightarrow G/K^+$ is a trivial $S^1$-bundle and $G/K^+$ is diffeomorphic to either a standard sphere $S^{2\ell_- + 1}$, the unit tangent bundle $T^1 S^{\ell_- +1}$, or to the Berger space $Sp(2)/Sp(1)_{max}$.

\end{proposition}

\begin{proof}
In Case (3) of Theorem \ref{thm:GH}, we have $\mathcal F \simeq_\Q S^1 \times S^{2\ell_- + 1} \times \Omega S^{2\ell_- + 3}$. Proposition \ref{prop:Connecting} implies that $\dim M = 2\ell_- + 3$ and moreover that the connecting homomorphisms $\partial_{2k+1}$ in the long exact rational homotopy sequence are zero if $2k+1 \neq 2\ell_- + 3$. It follows that $G/H$ has the rational homotopy groups of $S^1 \times S^{2\ell_- + 1}$.% \footnote{Since this space is intrinsically formal, $G/H$ is in fact rationally homotopy equivalent to $S^1 \times S^{2l_- + 1}$. Do we need this?}
%Going back to Grove-Halperin, we have the fibration $\mathcal{F}\rightarrow G/H\rightarrow M$.  They prove that $\pi_1(\mathcal{F})\cong \Z$, and the $\mathcal{F}$ has the rational type of $S^1\times S^{2\ell_- +1}\times \Omega S^{2\ell_- +3}$.  So, it follows that $\pi_1(G/H)\cong \Z$ and also that $2\ell_- + 3 = n$, and that $G/H$ is a rational $S^1\times S^{2\ell_-+1}$.\footnote{J:Grove-Halperin prove the regular leaf is always nilpotent, so $G/H$ has a Sullivan model which encodes its rational homotopy}

Now consider the bundle $S^1\rightarrow G/H\rightarrow G/K^+$.  Since $M$ is a rational sphere, the map $H^\ast(G/K^+;\Q)\rightarrow H^\ast(G/H;\Q)$ in the Mayer-Vietoris sequence associated to the double disk bundle decomposition of $M$ must be injective.  In particular, from the Gysin sequence, the rational Euler class $e\in H^2(G/K^+;\Q)$ must vanish.  Since $G/K^+$ is simply connected, $H^2(G/K^+;\Z)$ is torsion free, so it now follows that the integral Euler class must also vanish.  Thus, the bundle must be trivial.
%Since $G/K^+$ is simply connected, the bundle must be orientable.  In particular, it is a principal bundle.  Further, we already know from Mayer-Vietoris that the induced maps $H^*(G/K^\pm;\Q)\rightarrow H^\ast(G/H;\Q)$ are injective.  This implies the Euler class of the principal bundle vanishes\footnote{I use this argument in several places, but it has a gap:  the map $H^2(G/K^\pm;\Q)\rightarrow H^2(G/H;\Q)$ is injective, but it's not necessarily injective with integral coefficients.  The fix:  Since $G/K^+$ is simply connected, $H^2$ is torsion free.  Thus, the integral Euler class must vanish.  Now the rest of the argument works.}, so it is a product bundle: $G/H$ is diffeomorphic to $S^1\times G/K^+$.

Lastly, simply note that since $\pi^\Q_\ast(G/K^+\times S^1) \cong\pi^\Q_\ast(G/H)\cong \pi^\Q_\ast(S^{2\ell_- + 1}\times S^1)$, it easily follows that $G/K^+$ has the rational homotopy groups of $S^{2\ell_- + 1}$.  Since $G/K^+$ is simply connected and spheres are intrinsically formal, we must have $G/K^+\cong_\Q S^{2\ell_1 + 1}$.  Recalling that $\ell_-$ is odd, the diffeomorphism calculation now follows easily from the classification of biquotients with singly generated rational cohomology \cite{KapovitchZiller04}.
 \end{proof}

\begin{proposition}The space $G/K^-$ is diffeomorphic to the total space of a non-orientable $S^{l_- + 1}-$bundle over $S^1$.

\end{proposition}

\begin{proof}Consider the bundle, $S^{\ell_-}\rightarrow G/H\rightarrow G/K^-$, it follows from the long exact homotopy sequence that the rational homotopy groups of $G/K^-$ are supported in degrees $1,\ell_- +1, 2\ell_- +1$ and that $\pi_1(G/K^-)\cong \pi_1(G/H)\cong \pi_1(S^1\times G/K^-)\cong \mathbb{Z}$.  Note that the map $\pi_1(G/H)_\Q\rightarrow \pi_1(G/K^-)_\Q$ is an isomorphism, so the map $H^1(G/K^-;\Q)\rightarrow H^1(G/H;\Q)$ is as well.  It also follows that, up to cover, $G/K^-$ is a rational $S^1\times S^{\ell_- +1}$.  Since $\ell_-+1$ is even, from the classification of rational spheres which are homogeneous, we see that some cover of $G/K^-$ is diffeomorphic to $S^1\times S^{\ell_-+1}$.  Specifically, if $(K^-)^0\subseteq K^-$ is the identity component, then $G$ can be written as $G'\times S^1$, with $(K^-)^0$ embedded into $G'$ with $G'/(K^-)^0 = S^{\ell_- +1}$.  In particular, $G/K^-$ has the structure of bundle over $S^1$ with fiber a quotient of $S^{\ell_- + 1}$.

In fact, the fiber $F$ must be $S^{\ell_- + 1}$.  To see this, note that from the long exact sequence associated to $F\rightarrow G/K^-\rightarrow S^1$, $\pi_1(F)$ injects into $\pi_1(G/K^-)\cong\mathbb{Z}$.  Thus, $\pi_1(F)$ is finite and torsion free, so it must be trivial.

Finally, simply note that since the base space $S^1$ is orientable, if the bundle $S^{\ell_- + 1}\rightarrow G/K^-\rightarrow S^1$ is orientable, then so is the total space, contradicting the fact that we are in Case (3) of Theorem \ref{thm:GH}.
\end{proof}

\begin{lemma}\label{lem:xtop}  We have $H^\ast(G/K^-) \cong \begin{cases} \mathbb{Z} & \ast = 0,1\\ \mathbb{Z}/2\mathbb{Z} & \ast = \ell_- + 2\\ 0 & \text{other wise}\end{cases}$.  Further, $$H^\ast(G/K^-;\mathbb{Z}/2\mathbb{Z})\cong H^\ast(S^1\times S^{\ell_- + 1}; \mathbb{Z}/2\mathbb{Z})$$ and the Steenrod squaring operation $Sq^1:H^{\ell_- + 1}(G/K^-;\mathbb{Z}/2\mathbb{Z})\rightarrow H^{\ell_- + 2}(G/K^-;\mathbb{Z}/2\mathbb{Z})$ is an isomorphism.

\end{lemma}

\begin{proof}  The computation of $H^*(G/K^-)$ is a simple application of the Wang sequence, noting that non-orientability of the bundle implies the map $H^\ast(S^{\ell_- + 1})\rightarrow H^\ast(S^{\ell_- + 1})$, induced by the monodromy action, is multiplication by $-1$.  The computation with $\mathbb{Z}/2\mathbb{Z}$ coefficients is a simple application of the Universal Coefficients Theorem.

For the Steenrod power calculation, since $G/K^-$ is non-orientable, it follows that the first Stiefel-Whitney class $w_1\in H^1(G/K^-;\mathbb{Z}/2\mathbb{Z})$ is non-trivial.  This implies that the Wu class $v_1 = w_1$ is non-trivial.  Now, if $x\in H^{\ell_- + 1}(G/K^-;\mathbb{Z}/2\mathbb{Z})\cong \mathbb{Z}/2\mathbb{Z}$ is non-trivial, the Poincar\'e duality implies $w_1\cup x\neq 0$.  But then, by definition of the Wu class, $Sq^1(x) = w_1 \cup x\neq 0$.

\end{proof}

\begin{proposition}\label{prop:diff}  The space $G/K^+$ is diffeomorphic to $T^1 S^{\ell_- + 1}$.

\end{proposition}

\begin{proof}  From Proposition \ref{prop:gkh=1}, we only need to show that $G/K^+$ cannot be a sphere, nor the Berger space.

We begin by ruling out the Berger space using the Mayer-Vietoris sequence for the double disk bundle decomposition of $M$.  This space is a rational $7$-sphere whose cohomlogy ring contains $\mathbb{Z}/10\mathbb{Z}$ in $H^4$, but no other torsion \cite{KitchlooShankar01}.  From Proposition \ref{prop:gkh=1}, the map $H^4(G/K^+)\rightarrow H^4(G/H)$ is an isomorphism.  Since $H^4(G/K^-) = H^3(G/H) = 0$, the Mayer-Vietoris sequence implies $H^4(M) = 0$.  By Universal coefficients and Poincar\'e duality, $H^4(M)\cong H_3(M)\cong H^6(M)$, so $H^6(M) = 0$ as well.

Thus, the map $H^5(G/K^+)\oplus H^5(G/K^-)\rightarrow H^5(G/H)$ must be surjective.  Since $H^5(G/K^+) = 0$ and $H^5(G/K^-)\cong \mathbb{Z}/2\mathbb{Z}$, we find that $$\mathbb{Z}/2\mathbb{Z}\supset H^5(G/H)\cong H^1(S^1)\otimes H^4(G/K^+)\cong \mathbb{Z}/10\mathbb{Z},$$ giving a contradiction.

\bigskip

Next, we rule out the possibility that $G/K^+$ is a sphere.  To do so, we use the cohomological Serre Spectral sequence with $\mathbb{Z}/2\mathbb{Z}$ coefficients associated to the bundle $S^{\ell_-}\rightarrow G/H \rightarrow G/K^-$.  We claim that all differentials must vanish.  This is obviously the case except for the differentials on the $E_{\ell_- + 1}$ page.  Further, using the product structure on this page, it is enough to show that $d:E_{\ell_- + 1}^{0,\ell_-}\rightarrow E_{\ell_- +1}^{\ell_- + 1 ,0}$ is trivial.  To that end, assume it is non-trivial.  Then, from Lemma \ref{lem:xtop}, $Sq^1\circ d$ is also non-trivial.  As $Sq^1$ commutes with $d$ \cite[Corollary 6.9]{mccleary}, but $Sq^1$ is trivial on $H^\ast(S^{\ell_- + 1})$, we have a contradiction.

Thus, all differentials vanish, so $H^\ast(G/H;\mathbb{Z}/2\mathbb{Z})$ is isomorphic, as a group, to $H^\ast(S^{\ell_- + 1}\times S^1\times S^{\ell_-};\mathbb{Z}/2\mathbb{Z})$.  Since $G/H$ is diffeomorphic to $S^1\times G/K^+$, this is a contradiction to the assumption that $G/K^+$ is a sphere.

\end{proof}

We can now conclude the proof of Theorem \ref{thm:h=1}.

%\begin{proposition}\label{prop:nonorientcoh}
%Suppose a cohomogeneity one manifold $M^n$ with $n\geq 7$ is a simply connected rational sphere of odd dimension with exactly one non-orientable singular orbit $G/K^-$ of codimension $\ell_- + 1$.  If $G/K^+\cong T^1 S^{\ell_- + 1}$, then $M$ is homeomorphic to a sphere.
%\end{proposition}

\begin{proof}[Proof of Theorem \ref{thm:h=1}]
Note that $\dim M = 2\ell_- + 3$. Hence by Poincar\'e duality and the resolution of the Poincar\'e conjecture, it suffices to prove that the reduced cohomology groups vanish in degrees at most $\ell_- + 2$.

To do this, we use the Mayer-Vietoris sequence associated to the double disk bundle decomposition of $M$.  Propositions \ref{prop:gkh=1} and \ref{prop:diff} show that $G/K^+$ is diffeomorphic to $T^1 S^{\ell_- + 1}$ and $G/H$ is diffeomorphic to $S^1\times T^1 S^{\ell_- + 1}$.  Moreover, Lemma \ref{lem:xtop} computes the cohomology groups of $G/K^-$.  The only missing ingredient is knowledge of the maps in the Mayer-Vietoris sequence.  However, the projection $G/H\rightarrow G/K^+$ is bundle isomorphic to the trivial product $S^1\times G/K^+\rightarrow G/K^+$, while the isomorphism $\Z\cong \pi_1(G/H)\rightarrow \pi_1(G/K^-)$ implies the map $\Z\cong H^1(G/K^-) \rightarrow H^1(G/H)$ is an isomorphism, it easily follows that $H^\ast(M) = 0$ for $\ast \leq \ell_- + 1$.  Moreover, if $H^{\ell_- + 2}(M)\neq 0$, then it follows that $H^{\ell_- + 3}(M)\cong H^{\ell_- + 2}(G/H)\cong \mathbb{Z}/2\mathbb{Z}$.  But then Poincar\'e duality and Universal coefficients give $$\mathbb{Z}/2\mathbb{Z}\cong H^{\ell_- + 3}(M)\cong H_{\ell_-}(M)\cong H^{\ell_- + 1}(M) = 0,$$ giving a contradiction.
\end{proof}

%\begin{proposition}	Suppose a cohomogeneity one manifold $M^n$ with $n\geq 7$ is a simply connected rational sphere of odd dimension with exactly one non-orientable singular orbit $G/K^-$ of codimension $\ell_- + 1$.  Then $G/K^-$ is diffeomorphic to the total space of the unique non-trivial $S^{\ell_- + 1}-$ bundle over $S^1$.  Further, the bundle $G/H\rightarrow G/K^+$ is isomorphic to the trivial bundle $S^1\times G/K^+\rightarrow G/K^+$ and $G/K^+$ is a simply connected rational sphere of dimension $2\ell_- + 1$.  \end{proposition}
%Since $G/K^+$ is simply connected, it now follows that $H^2(G/H;\Z)$ is trivial.  In more detail, since the cohomology groups of $S^1$ are free, Kunneth tells us $H^2(G/H;\Z)\cong H^2(S^1)\oplus H^1(S^1)\otimes H^1(G/K^+)\oplus H^2(G/K^+)$.  Of course, $H^2(S^1) = 0$.  Further, $H^1(G/K^+) = 0$ since $G/K^+$ is simply connected.  Finally, the free part of $H^2(G/K^+)$ vanishes since $G/K^+$ is a rational $S^{2\ell_- +1}$, and the torsion part of $H^2(G/K^+)$ vanishes because it is simply connected.

%\begin{remark}  The Kervaire spheres all arise in this case as Brieskorn varieties.  See GVWZ. \end{remark}

\section{Case (5) and (6) of Theorem \ref{thm:GH}}\label{sec:exceptional}

In this section, recalling Convention \ref{convention}, we prove the following theorem.

\begin{theorem}\label{thm:exceptional}  Suppose $M^n$ is a rational sphere, and a compact Lie group $G$ acts on $M$ via a cohomogeneity one action.  If the homotopy fiber of the inclusion of a principal orbit is either in case (5) or case (6) in Theorem \ref{thm:GH}, then $M$ is equivariantly diffeomorphic to a standard sphere equipped with a linear action.

\end{theorem}

We quickly prove case (5), where $\ell_- = \ell_+$ is even and $\mathcal{F}$ is rationally $S^{\ell_- }\times \Omega S^{\ell_- + 1}$.

\begin{proposition}  Suppose $M^n$ is a rational sphere, and a compact Lie group $G$ acts on $M$ via a cohomogeneity one action.  If the homotopy fiber of the inclusion of a principal orbit belongs to case (5) in Theorem \ref{thm:GH}, then $M$ is equivariantly diffeomorphic to a standard sphere equipped with a linear action having two fixed points.

\end{proposition}

\begin{proof}  Acording to \cite[Propositions 1.22 and 1.23]{Hoelscher10}, a cohomogeneity one action with two fixed points is equivalent to an isometric action on a standard sphere.  So we need only show the $G$ action has two fixed points.

From Proposition \ref{prop:Connecting}, we see that $n -1 = \ell_-$. Thus we have $\dim G/K^\pm = \dim G/H - \ell_\pm = 0$.  Since $G$ is connected, it follows that both $G/K^\pm$ are points.

\end{proof}

For the remainder of this section, we assume Case (6) of Theorem \ref{thm:GH} occurs.  Thus, $\ell_- = \ell_+ \in \{2,4,8\}$ and $\mathcal{F}$ has the rational type of one of $$ (SU(3)/T^2)\times \Omega S^7, (Sp(2)/T^2)\times \Omega S^9, (\mathbf{G}_2/T^2)\times \Omega S^{13}, $$ $$ (Sp(3)/Sp(1)^3) \times \Omega S^{13}, (\mathbf{F}_4/Spin(8))\times \Omega S^{25}.$$   We will prove Theorem \ref{thm:exceptional} by reducing it to a theorem of Wang \cite{Wang60}[Theorem V].  In the proof of this theorem, Wang considers five possibilities for $(G,H)$: $$(G,H)\in \{(SU(3),T^2), (Sp(2),T^2), (\mathbf{G}_2,T^2), (Sp(3), Sp(1)^3), (\mathbf{F}_4, Spin(8))\}.$$  Note that for each of these five possibilities for $(G,H)$, the values of $\ell_-=\ell_+$ constrain $K^\pm$ to both be locally isomorphic to $S^1\times SU(2)$ for the first three possibilities, while they are both locally isomorphic to $Sp(2)\times Sp(1)$ and $Spin(9)$ in the fourth and fifth cases, respectively.

For each of the five pairs, Wang  shows that the assumption that $M$ is a standard sphere implies there is, up to equivalence, a unique cohomogeneity one $G$ action on $M$ with principal isotropy group $H$.  For a given $(G,H)$, there are, of course, several inequivalent cohomogeneity one manifolds coming from different possible embeddings of $K^\pm$, but Wang uses just two propositions \cite[(6.2) and (7.5)]{Wang60} to rule out any actions which are not on standard spheres.  Thus, Case (6) of Theorem \ref{thm:GH} will be handled if we can prove that both \cite[(6.2) and (7.5)]{Wang60} are valid in the case of rational spheres, and if we can show that the set up of Case (6) implies $(G,H)$ must be one of these five possibilities.

Wang's \cite[(6.2)]{Wang60}, generalized to rational spheres and adapted to our notation, is as follows:

\begin{proposition}  Suppose a cohomogeneity one $G$-manifold $M$ has group diagram $H\subseteq K^\pm \subseteq G$ and that  both $G/K^\pm$ have even codimension in $M$.  If $M$ is a rational sphere, then $K^+$ and $K^-$ generate $G$.

\end{proposition}

This proposition is, in fact, valid even without the restriction on the parity of the codimensions of $G/K^\pm$.  This is an immediate consequence of Theorem \ref{thm:primitive} above, for if the group generated by $K^+$ and $K^-$ is strictly contained in $G$, then the action is primitive.

Wang's \cite[(7.5)]{Wang60}, adapted to the case of rational spheres, is as follows:

\begin{proposition}\label{prop:wang75}Suppose a cohomogeneity one $G$-manifold $M$ has group diagram $T^2\subseteq U(2), U(2)\subseteq G$ with $G$ simple, simply connected, and rank $2$.  If for both $K^\pm$, the center $Z(K^\pm)$ is contained in the semi-simple part of $K^\mp$, then $M$ is not a rational sphere.

\end{proposition}

\begin{proof} Using the Mayer-Vietoris sequence for the double disk bundle decomposition of $M$, it is enough to show that the map $H^4(G/K^+;\mathbb{Q})\oplus H^4(G/K^-;\mathbb{Q}) \rightarrow H^4(G/H;\mathbb{Q})$ has non-trivial kernel.  In other words, it is sufficient to show the two maps $H^4(G/K^\pm;\Q)\rightarrow H^4(G/H;\Q)$ have images which intersect non-trivially.

 We first note that the group generated by the centers of $K^\pm$, $Z(K^+)Z(K^-)$, is a subgroup of $H$.  Indeed, passing to the two fold cover $SU(2)\times S^1$ of $K^+$, $H$, being full rank, must lift to a subgroup of the form $S^1\times S^1$ with the first $S^1\subseteq SU(2)$.  The projection of the second $S^1$ is precisely $Z(K^+)$, so $Z(K^+)\subseteq H$.  Analogously, one sees that $Z(K^-)\subseteq H$, so $Z(K^+)Z(K^-)\subseteq H$.  The assumption that the center of $K^+$ is contained in the semi-simple part of $K^-$ implies that $Z(K^+)\neq Z(K^-)$, from which it follows that $H = Z(K^+)Z(K^-)$.

Writing $BL$ for the classifing space of a Lie group $L$, the fact that each of  $H,K^\pm$ have full rank in $G$, implies via \cite[Theorem 6.5]{Singhof}  that we have an isomorphism $$H^\ast(G/H;\mathbb{Q})\cong H^\ast(BH;\mathbb{Q})/H^\ast(BG;\mathbb{Q}),$$ and likewise with $H$ replaced by $K^\pm$.  Here, $H^\ast(BG;\mathbb{Q})$ is understood as a subalgebra of $H^\ast(BH;\mathbb{Q})$ with inclusion induced from the inclusion $H\rightarrow G$.  Thus, it is sufficient to show the two maps $H^4(BK^\pm;\Q) \rightarrow H^4(BH;\Q)/H^4(BG;\Q)$ have images which intersect non-trivially.  Because $\dim H^4(BG;\Q) = 1$, as it is for any simple Lie group, it is sufficient to show the two maps $H^4(BK^\pm;\Q)\rightarrow H^4(BH;\Q)$ have images which intersect in a two-dimensional subspace.

Since $H = T^2$, we may identify $H^\ast(BH;\Q)\cong \Q[z_1,z_2]$ where both $|z_i|=2$ and we may likewise identify $H^\ast(BK^\pm;\Q)\cong \Q[z_{\pm}, t_{\pm}]$ with both $|z_{\pm}| = 2$ and $|t_{\pm} | = 4$.  Since $H = Z(K^+)Z(K^-)$, it follows that up to permuting $z_1$ and $z_2$, the inclusion $H\rightarrow K^+$ induces the map $z_+\mapsto z_1$, and similarly, $z_-\mapsto z_2$.  The hypothesis that $Z(K^\pm)$ is a subgroup of the semi-simple part of $K^\mp$ implies that $t_+$ maps to $z_2^2$ and $t_-$ maps to $z_1^2$.  Thus, the two maps $H^\ast(BK^\pm;\Q)\rightarrow H^\ast(BH;\Q)$ intersect in the two dimensional subspace of $H^4(BH;\Q)$ spanned by $z_1^2$ and $z_2^2$, completing the proof.
\end{proof}

To complete the proof of Theorem \ref{thm:exceptional}, we need only show demonstrate that the only possibilities for $(G,H)$ which can arise in Case (6) of Theorem \ref{thm:GH} are the five considered by Wang.  We use the notation $\mathcal{F}'$ to denote one of the elements of $\{SU(3)/T^2, Sp(2)/T^2, \mathbf{G}_2/T^2, Sp(3)/Sp(1)^3, \mathbf{F}_4/Spin(8)\}$.  From Proposition \ref{prop:Connecting}, it follows that in case (6) of Theorem \ref{thm:GH} that $G/H$ has the same rational homotopy groups as some $\mathcal{F}'$.

\begin{proposition}  If a compact Lie group $G$ acts by cohomogeneity one on a rational sphere $M$, then the rational Pontryagin classes of the principal orbit $G/H$ must vanish.

\end{proposition}

\begin{proof}Let $i:G/H\rightarrow M$ be the inclusion.  Since $G/H$ is the boundary of one of the disk bundles whose union is $M$, $G/H$ has a collar, so $i^\ast(TM) = T(G/H)\oplus 1$ with $1$ denoting the trivial line bundle over $G/H$.  Then $i^\ast(p(TM)) = p(i^\ast(TM)) = p(T(G/H))$.  However, since $M$ is a rational sphere of odd dimension, $p(TM) = 1\in H^\ast(M;\mathbb{Q})$.

\end{proof}

As a consequence of this proposition, it follows that in case $(6)$, that  $G/H$ cannot be a product of two or more homogeneous spaces of positive dimensions.  For, by inspection of the rational homotopy groups, if $G/H$ were a product, it would be a product $S^{\ell_-}\times X$ where $X$ has singly generated rational cohomology and where $X$ is not a rational sphere.  From the classification of biquotients with singly generated rational cohomology \cite{KapovitchZiller04}, one can see that a homogeneous rational projective space always has non-zero rational Pontryagin classes and thus, so does the product.

It now follows that $G$ must be simple.  For, if $G = G_1\times G_2$, then $H = H_1\times H_2$ with each $H_i\subseteq G_i$ full rank \cite[pg. 205]{BorelDeSiebenthal49}.  But then $G/H\cong (G_1/H_1)\times (G_2/H_2)$ is a product.  Both $G_1/H_1$ and $G_2/H_2$ have positive dimension, unless $H_1 = G_1$ or $H_2 = G_2$.  But if either of these cases occurs, then $G$ contains a normal subgroup acting with the same orbits, contrary to Convention \ref{convention}.

We are now ready to compute the possibilities for $(G,H)$, completing the analysis of Case (6) of Theorem \ref{thm:GH} and proving Theorem \ref{thm:exceptional}.
  
  \begin{proposition}\label{prop:gtype}Suppose a simple simply connected compact Lie group $G$ acts via a cohomogeneity one action on a rational sphere $M$.  If the non-loop space factor $\mathcal{F}'$ of the homotopy fiber of the inclusion $G/H\rightarrow M$ has the same rational homotopy groups as an element of $\mathcal{F}' \in \{SU(3)/T^2, Sp(2)/T^2, \mathbf{G}_2/T^2, Sp(3)/Sp(1)^3, \mathbf{F}_4/Spin(8)\}$, then $(G,H)$ is one of $$\{(SU(3),T^2), (Sp(2), T^2), (\mathbf{G}_2,T^2), (Sp(3),Sp(1)^3), (\mathbf{F}_4, Spin(8))\}.$$
  
  \end{proposition}
  
  \begin{proof}  In all cases, since $\dim \pi^\Q_{even}(\mathcal{F}') = \dim^\Q \pi_{odd}(\mathcal{F}')$, it follows that $G$ and $H$ must have the same rank.

  Now, initially assume $\mathcal{F}'\in \{SU(3)/T^2, Sp(2)/T^2, \mathbf{G}_2/T^2\}$.  Because $\pi^\Q_3(\mathcal{F}')\neq 0$ and $G$ is simple, it follows that $H$ is a torus.  Since $\dim\pi^\Q_2(\mathcal{F}') = 2$, $H = T^2$ and thus $G$ has rank $2$.  Then counting dimensions and using the classification of simple Lie groups gives the result in this case.
  
  Next, assume $\mathcal{F}' = Sp(3)/Sp(1)^3$.  Since $\pi_2^\Q(\mathcal{F}') = 0$, $H$ is semi-simple.  Since $\dim \pi_4^\Q(\mathcal{F}')-\dim \pi^\Q_3(\mathcal{F}') = 2$, $H$ must have three simple factors.  As $H$ is the isotropy group of the transitive $K^+$ action on $S^4$, $H$ must be locally isomorphic to $Sp(1)^2\times H'$ for some simple Lie group $H'$, with $H'$ acting trivially on $K^+/H$.  The $G$-action is primitive (Theorem \ref{thm:primitive}), so from Proposition \ref{prop:finiteintersection},  the isotropy action of $H'$ on $S^4 = K^-/H$ must be non-trivial, so $H'$ maps non-trivially into $SO(4)$.  Since $H'$ is simple, this forces $H' = Sp(1)$, so $H$ is locally isomorphic to $Sp(1)^3$ and thus, $K^\pm$ are both locally isomorphic to $Sp(2)\times Sp(1)$.  Counting dimensions now shows that $G = Sp(3)$ or $G = Spin(7)$.  However, $K^+$ is covered by $Sp(2)\times Sp(1)$, and elementary representation theory indicates the smallest orthogonal representation of $Sp(2)\times Sp(1)$ has dimension $8$.  Thus, $G= Sp(3)$, which then implies that $K^\pm$ are both isomorphic to $Sp(2)\times Sp(1)$, which then implies that $H$ is isomorphic to $Sp(1)^3$.

  Lastly, assume $\mathcal{F}' = \mathbf{F}_4/Spin(8)$.  Since $\pi_3^\Q(\mathcal{F}') = \pi_4^\Q(\mathcal{F}') = 0$, it follows that $H$ must be simple.  Further, $H$ must be the isotropy group of a transitive action on $S^8$, so $H$ is, up to cover, $Spin(8)$ and both $K^\pm$ are up to cover $Spin(9)$.  Counting dimensions now easily gives $G = \mathbf{F}_4$.  Further, from \cite{BorelDeSiebenthal49}, it follows that $K^\pm$ are both isomorphic to $Spin(9)$, which implies that $H$ is isomorphic to $Spin(8)$.
  
  \end{proof}

\section{Case (4) of Theorem \ref{thm:GH}}\label{sec:last}

In this section, we continue to follow Convention \ref{convention}.  We assume that the homotopy fiber of the inclusion of a principal orbit $G/H$ into $M$ belongs to Case (4) of Theorem \ref{thm:GH}.  Thus, $\mathcal{F} \simeq_\Q S^{\ell_-} \times S^{\ell_+} \times \Omega S^{\ell_- + \ell_+ + 1}$. 

We will first show that if $\ell_-$ and $\ell_+$ have the same parity, then $M$ must be homeomorphic to a sphere.  We first need a lemma.

\begin{lemma}\label{lem:intersection}  Suppose $M^n$ is a rational sphere and that a compact Lie group $G$ acts on $M$ via a cohomogeneity one action.  If the action has no fixed points then $(K^+\cap K^-)/H$ is finite.

\end{lemma}

\begin{proof}   Assume the $G$ action has no fixed points, so both $K^\pm$ are proper subgroups of $G$.  Set $K = K^+\cap K^-$ and assume for a contradiction that $K$ has positive dimension.  If, say, $\ell_+ = 1$, then $K$ must be a union of components of $K^+$.  However, from the surjective map $S^{\ell_+}\cong K^+/H\rightarrow K^+/K$, we see that $K^+/K$ is connected, so $K = K^+$.  But then the action is non-primitive, with $K^\pm \subseteq K^-\subsetneq G$.  This contradicts Theorem \ref{thm:primitive}.  Thus we must have both $\ell_\pm > 1$.  In particular, $H$ and $K^\pm$ are all connected and $K$ is a proper subgroup of both $K^\pm$.

As $H$ is connected, we have $H\subseteq K^0\subseteq K^\pm$, with $K^0$ denoting the identity component of $K$.  Thus we have a homogeneous fiber bundle $K^0/H\rightarrow K^\pm/H \cong S^{\ell_\pm}\rightarrow K^\pm/K^0$ whose total space is a sphere $S^{\ell_\pm}$.  Thus, we deduce that $K^0/H$ is a homotopy $S^1,S^3,$ or $S^7$ \cite{Browder62}, and hence diffeomorphic to a sphere \cite[4.6]{Borel532}.

 Now, consider the homogeneous fibration associated to the triple $H\subseteq K^0\subseteq G$,  $K^0/H\rightarrow G/H\rightarrow G/K^0$.   We claim that the induced map $H^k(G/K^0;\Q)\rightarrow H^k(G/H;\Q)$ is not surjective for at least one $k< \dim G/H$. 

Momentarily believing this claim, note that both maps $\pi_{\pm}:G/H\rightarrow G/K^\pm$ factor through $G/K^0$.  It would then follow that in the Mayer-Vietoris sequence for the double disk bundle decomposition of $M$, the map $H^k(G/K^+;\Q)\oplus H^k(G/K^-;\Q)\rightarrow H^k(G/H;\Q)$ fails to be surjective.  This contradicts the fact that $M$ is a rational sphere.

 %Set $K = K^+\cap K^-$ and assume for a contradiction that $K/H$ has positive dimension.  Then chain of subgroups $H\subseteq K\subseteq K^\pm$ gives rise to a homogeneous fibration $K/H\rightarrow K^\pm/H\rightarrow K^\pm/K$.   If $\rk K = \rk H$, the fiber has positive Euler characteristic, which implies that at least one even degree rational homotopy group of $K/H$ is non-trivial.  From \cite[Corollary 3.5]{DeVitoKennard20}, we must have both $\ell_\pm$ even.  However, as even spheres are not the total space of any non-trivial bundle with positive dimensional fiber, we deduce that both $K^\pm/K$ are points.  That is, $K^+ = K^- = K$.  In particular, the action is not primitive unless it has a fixed point, giving a contradiction.

%Hence we may assume $\rk K > \rk H$, so $\rk K^+ = \rk K^- = \rk K = \rk H + 1$, and so both $\ell_\pm$ are odd.  If, say, $\ell_+ = 1$, then $K^+/K$ is connected and zero-dimensional; that is, that $K^+ = K\subseteq K^-$.  Thus, we have again contradicted Theorem \ref{thm:primitive}, so we must have both $\ell_\pm \geq 3$.  In particular, from Proposition \ref{prop:ellinfo}, $H$ and both $K^\pm$ are connected and so $H\subseteq K^0$.  Note further that, since $\rk K^\pm = \rk K$, both $K^\pm/K^0$ are simply connected.

%Now, considering the homogeneous fibration associated to $H\subseteq K^0\subseteq K^\pm$,  since the total space is a sphere of odd dimension, we conclude that both $K^\pm/K^0$ are compact rank one symmetric spaces of even dimension and  that $K^0/H$ is a sphere $S^1,S^3$ or $S^7$.

Thus, we need only demonstrate the claim.   Since both $G/H$ and $G/K^0$ are orientable, this sphere bundle has an Euler class $e\in H^s(G/K^0;\Q)$ for $s\in{2,4,8}$.  If the claim is false, then it follows from the spectral sequence that cupping with $e$, thought of as a map $H^t(G/K^0;\Q)\rightarrow H^{t+s}(G/K^0;\Q)$, must be injective for any $t< \dim G/H$.  It follows easily from this that $G/K^0$ must have singly generated rational cohomology, and hence, that $G/H$ is a rational sphere.

Now, the bundles $S^{\ell_\pm}\rightarrow G/H\rightarrow G/K^\pm$ must have trivial rational Euler classes, for otherwise one of the maps $H^{\ell_\pm + 1}(G/K^+;\Q)\oplus H^{\ell_\pm + 1}(G/K^-)\rightarrow H^{\ell_\pm + 1}(G/H;\Q)$ in the Mayer-Vietoris sequence for the double disk bundle decomposition of $M$ has non-trivial kernel, contradicting the fact that $M$ is a rational sphere.  Since $G/H$ is a rational sphere, it now follows that both $G/K^\pm$ are points.  In other words, the $G$ action has a fixed point, giving a contradiction, and establishing the claim.

\end{proof}

\begin{proposition}Suppose $M^n$ is a cohomogeneity one rational sphere with homotopy fiber $\mathcal{F}$ belonging to Case (4) of Theorem \ref{thm:GH}.  If $\ell_-$ and $\ell_+$ have the same parity, then $M$ is homeomorphic to a sphere.

\end{proposition}

\begin{proof}  If the $G$ action has a fixed point, we are done from \cite{Hoelscher10}[Proposition 1.28].  So, we assume the $G$ action has no fixed points.  We will follow an approach of Wang [(4.7) and (4.8)]\cite{Wang60}.  Consider the map $f:G/H\rightarrow G/K^+\times G/K^-$ with $f(gH) = (gK^+, gK^-)$.  We will show that $f$ is a diffeomorphism.

To begin with, since $\ell_-$ and $\ell_+$ have the same parity, Proposition \ref{prop:Connecting} together with the fact that $\pi_{2k}(\Omega S^{\ell_- + \ell_+ + 1})\otimes \Q$ is supported in degree $2k = \ell_- + \ell_+$, we conclude $n = \ell_- + \ell_+ -1$.  Thus, $\dim G/K^+\times G/K^- = \dim G/H$.  It now follows that $f$ is a covering map.  Indeed, being $G$-equivariant, $f$ must have constant rank, but it must have full rank at some point by Sard's theorem.  Hence, by Ehresmann's theorem \cite{Ehresmann}, $f$ is a fiber bundle map.  Now, suppose the $f(eH) = f(gH)$ for some $g\in G$.  Then $gK^+ = eK^+$ and $gK^- = eK^-$, that is, $g\in K^+\cap K^-$, so the fiber is contained in $(K^+\cap K^-)/H$.  Since the action has no fixed points, Lemma \ref{lem:intersection} shows that $(K^+\cap K^-)/H$ is finite.  Thus, the fiber is finite and so $f$ is a covering map.

By the classification of covering maps, if we can show $f$ induces a surjection on fundamental groups, it will then follow that $f$ is a diffeomorphism.  Because we are in Case (4) of Theorem \ref{thm:GH}, the homotopy fiber $\mathcal{F}$ has abelian fundamental group, which implies the fundamentals groups of $G/H$ and both $G/K^\pm$ are abelian.  In particular, we may identify $\pi_1$ with $H_1$ for $G/H$ and both $G/K^\pm$.

Using the K\:unneth Theorem to identify $H_1(G/K^+)\times H_1(G/K^-))\cong H_1(G/K^+)\oplus H_1(G/K^-)$, it is not too hard to see that the composition $H_1(G/H)\rightarrow H_1(G/K^+)\oplus H_1(G/K^-)$ is the map appearing in the Mayer-Vietoris sequence for the double disk bundle decomposition of $M$.  In particular, since $M$ is simply connected, this composition is surjective.  It follows that $f_\ast$ is surjective on fundamental groups.  As $f$ is a covering, we therefore conclude it is a diffeomorphism.

To complete the proof, simply observe that  for either choice of $\pm$, $f$ gives a bundle isomorphism covering the identity  between the bundle $S^{\ell_\pm}\rightarrow G/H\rightarrow G/K^\pm$ and $G/K^\mp\rightarrow G/K^+\times G/K^-\rightarrow G/K^\pm$.  In particular, $G/K^\pm$ must be diffeomorphic to $S^{\ell_\mp}$, $G/H$ is diffeomorphic to $S^{\ell_+}\times S^{\ell_-}$, and the projection maps $G/H\rightarrow G/K^\pm$ are the obvious ones.

Now a simple application of Mayer-Vietoris shows that $M$ is an integral homology sphere, and hence homeomorphic to a sphere by the resolution of the Poincar\'e conjecture.

\end{proof}

We henceforth assume that $\ell_-$ and $\ell_+$ have different parities.  We will prove:

\begin{theorem}\label{thm:paritycase}  Suppose $M^n$ is a rational sphere and that $G$ is a compact Lie group acting via a cohomogeneity one action on $M$ with homotopy fiber $\mathcal{F}$ belonging to Case (4) of Theorem \ref{thm:GH}.  Assume in addition that $\ell_-$ and $\ell_+$ have different parities.  Then the action is orbit equivalent to one of the following.

\begin{itemize}\item a linear action on a standard sphere

\item  The usual action by $SO(2)\times SO(2k)$ on a Brieskorn variety $B^{4k-1}_d$, or, when $k = 4$, the restriction of this action to $SO(2)\times Spin(7)\subseteq SO(2)\times SO(8)$.

\item  An action of $SU(3)\times SU(2)$ on one of two homogeneous spaces of the form $\mathbf{G}_2/SU(2)$.
\end{itemize}

\end{theorem}

We begin by computing the rational topology of the principal and singular leaves.

\begin{proposition}\label{prop:lasttop}Suppose $M^n$ is a rational sphere and that $G$ is a compact Lie group acting via a cohomogeneity one action on $M$ with homotopy fiber $\mathcal{F}$ belonging to Case (4) of Theorem \ref{thm:GH}.  If $\ell_-$ and $\ell_+$ have different parities, then $n = 2(\ell_- + \ell_+) + 1$ and there are isomorphisms $$H^\ast(G/H;\Q)\cong H^\ast(S^{\ell_-}\times S^{\ell_+}\times S^{\ell_- + \ell_+};\Q), H^\ast(G/K^\pm;\Q)\cong H^\ast(S^{\ell_\mp}\times S^{\ell_- + \ell_+};\Q).$$  Further, there are also corresponding isomorphisms between the rational homotopy groups.

\end{proposition}

\begin{proof}Because $\ell_-$ and $\ell_+$ have different parities, $\pi_{2k}(\Omega S^{\ell_- + \ell_+ + 1})$ non-trivial when $2k = 2(\ell_- + \ell_+)$.  Now, Proposition \ref{prop:Connecting} implies that $n-1 = 2(\ell_- + \ell_+) $, i.e., that $n = 2(\ell_- + \ell_+) + 1$.

For the rational cohomology, we use the rational Serre spectral sequence for the fibration $\mathcal{F}\rightarrow G/H\rightarrow M$.  Since $M$ is a rational sphere, all differentials with domain $E_p^{0,k}$ for $k\leq n-1$ vanish for trivial reasons.  Since $\dim G/H = n-1$, this completely determines the rational cohomology of $G/H$.

%look as the relative Sullivan model for the fibration $\mathcal{F}\rightarrow G/H\rightarrow M$.  This gives a Sullivan model for $G/H$ in terms of a models $(\Lambda F, d_F)$ and $(\Lambda M, d_M)$ for $M$.  Specifically, one has a model for $G/H$ of the form $(\Lambda F\otimes \Lambda M, d)$ where $d|_{1\otimes \Lambda  M} = d_M$ and $d|_{\Lambda F \otimes 1}- d_F \in \Lambda F\otimes \Lambda^+ M$, with $\Lambda^+ M$ referring to the elements of $\Lambda M$ with positive degree.

%As $M$ is a rational sphere of dimension $n = 2(\ell_- + \ell_+) + 1$, we choose to take $\Lambda M = \Lambda[x]$ with $|x| = 2(\ell_-+\ell_+)+1$ and trivial differential.  Hence, any element of $\Lambda^+ M$ has degree $|x|$.  Because $|x| > 2\ell_\pm -1\geq \ell_\pm$ and $|x| > \ell_- + \ell_+$, it follows that for elements of $\Lambda F$ of degree at most $max( 2\ell_\pm -1)$, $d = d_F$.  In particular, in degrees up to $2(\ell_- + \ell_+)$, $H^\ast(G/H;\Q)$ must agree with $H^\ast(\mathcal{F};\Q)$.  As $\dim G/H = n-1 = 2(\ell_- + \ell_+)$, this, in fact, determines the entire rational cohomology ring of $G/H$.

For the rational cohomology of $G/K^\pm$, note that the the maps $H^\ast(G/K^\pm;\Q)\rightarrow H^\ast(G/H;\Q)$ must be injections, for otherwise the Mayer-Vietoris sequence associated to the double disk bundle decomposition of $M$ would show that $M$ is not a rational sphere.  Thus, in rational Serre spectral sequence for the sphere bundles $G/H\rightarrow G/K^\pm$, we see the all differentials must vanish.  It follows that $H^\ast(G/H;\Q)\cong H^\ast(S^{\ell_\pm}\times G/K^\pm;\Q)$, from which the rational cohomology ring of $G/K^\pm$ is easily determined.

Lastly, we prove the statement about rational homotopy groups.  For the case of a product of two spheres, it is easy to compute the minimal Sullivan model directly, giving the result.  For $G/H$, it follows that the map $H_{\ell_-}(S^{\ell_-};\Q)\rightarrow H_{\ell_-}(G/H;\Q)$ is injective, so the Hurewicz homomorphism $\pi_{\ell_-}(G/H)\otimes \Q\rightarrow H_{\ell_-}(G/H;\Q)$ is non-trivial.  Then \cite[Corollary 2.78]{FelixOpreaTanre08} shows that $G/H$ has the rational homotopy type of $S^{\ell_-} \times G/K^-$, so the conclusion follows.

\end{proof}

We now use Proposition \ref{prop:lasttop} to better understand the structure of $H,K^\pm,$ and $G$.

\begin{proposition}\label{prop:gstructure}  Suppose a compact connected Lie group $G$ acts via an irreducible cohomogeneity one action on a rational sphere $M^n$ with the homotopy fiber of the inclusion of a principal orbit $G/H$ into $M$ falling into Case (4) of Theorem \ref{thm:GH}.  Assume $\ell_-$ is odd and $\ell_+$ is even.  Then $G$ is semi-simple unless $\ell_- = 1$, in which case $G$ has a one dimensional torus factor.  Further, $K^+$ is always semi-simple, and $H$ is semi-simple unless $\ell_+ = 2$, in which case $H$ has a one dimensional torus factor.  Lastly, $K^-$ is semi-simple if both $\ell_- > 1$ and $\ell_+ > 2$, has a two dimensional torus factor if $\ell_- = 1$ and $\ell_+ =2$, and has a one dimensional torus factor otherwise.

\end{proposition}

\begin{proof}For the claims about $G$ and $H$ we argue as follows.  Writing $G = G_1\times T^k$, irreducibility implies that $H^0\subseteq G_1$.  In particular, the map $\pi^\Q_1(H)\rightarrow \pi^\Q_1(G)$ is trivial.  Since $H$ has finitely many components and $\pi_2(G) = 0$, the long exact sequence in rational homotopy groups associated to the bundle $H\rightarrow G\rightarrow G/H$ now gives isomorphisms $\pi^\Q_1(G)\cong \pi^\Q_1(G/H)$ and $\pi^\Q_1(H)\cong\pi^\Q_2(G/H)$.  The result now follows from Proposition \ref{prop:lasttop}.

To see that $K^+$ is always semi-simple, note that from Proposition \ref{prop:lasttop}, $G/K^+$ has the rational homotopy groups of a product of two odd dimensional spheres.  In particular, $\pi^\Q_{even}(G/K^+)$ is trivial.  Now, assume for a contradiction that $\pi^\Q_1(K^+)$ is non-trivial.  Then the long exact sequence associated to the bundle $K^+\rightarrow G\rightarrow G/K^+$ shows that $\pi^\Q_1(K^+)\rightarrow \pi^\Q_1(G)$ is injective.  On the other hand as $K^+$ has finitely many components, so the map $\pi_1^\Q(G)\rightarrow \pi_1^\Q(G/K^+)$ is a surjective map between vector spaces of the same dimension, so is an isomorphism.  Thus, exactness forces $\pi_1^\Q(K^+) = 0$, so $K^+$ is semi-simple.

Lastly, for $K^-$, from Proposition \ref{prop:ellinfo}, if $\ell_- = 1$, then $K^- = H\cdot S^1$, so $K^-$ has one more circle factor than $H$.  On the other hand, if $\ell_- > 1$, then $G$ is semi-simple, so the long exact sequence associated to $K^-\rightarrow G\rightarrow G/K^-$ gives $\pi_2(G/K^-)\otimes \Q\cong \pi_1(K^-)\otimes \Q$.  The result now follows since $G/K^-$ has the rational homotopy groups of $S^{\ell_+}\times S^{\ell_- + \ell_+}.$

\end{proof}

We now break the proof of Theorem \ref{thm:paritycase} into three cases.  First, we assume $G$ is simple and find there there are precisely two such actions, both of which are linear actions on standard spheres.  Next, we assume $G$ has a circle factor.  Here we show such actions are precisely the Brieskorn actions.  Lastly, we assume $G$ is semi-simple but not simple.  In this case, we find two infinite families of linear actions on spheres and two sporadic examples on rational $11$-spheres.

\subsection{$G$ has one simple factor}

In this subsection, in addition to Convention \ref{convention}, we will always assume that $\ell_-$ is odd, $\ell_+$ is even, and that $G$ is simple.  We prove the following.

\begin{theorem}\label{thm:gsimp} Suppose $M$ is a rational sphere upon which a simple compact Lie group $G$ acts via a cohomogeneity one action.  Suppose further that the homotopy fiber of the inclusion of a principal orbit into $M$ falls within Case (4) of Theorem \ref{thm:GH}, and that $\ell_-$ and $\ell_+$ have different parities.  Then either $G = SU(5)$ and $M$ is equivariantly diffeomorphic to $S^{19}$ with $G$ acting via the representation $ \Lambda^2 \mathbb{C}^5$, or $G = Spin(10)$ and $M$ is equivariantly diffeomorphic to $S^{31}$ with $G$ acting via the spin representation.
\end{theorem}

From Proposition \ref{prop:lasttop}, $\pi_{\ast}^\Q(G/K^+)$ has dimension $2$ and is concentrated in odd degrees.  If follows that $K^+$ has corank $2$ in $G$ and that $K^+$ is simple and non-abelian.  As $K^+/H\cong S^{\ell_+}$, Table \ref{table:transsphere} now implies that, up to cover, $K^+ = Spin(\ell_+ + 1)$ or, when $\ell_+ = 6$, that $K^+ = \mathbf{G}_2$.

We now turn to Table \ref{table:corank2}.  Specifically, we find all entries where $L$ is, up to cover, an odd Spin group or $\mathbf{G}_2$, keeping in mind that $Spin(3) = Sp(1) = SU(2)$ and $Spin(5) = Sp(2)$.   We discard any examples where $L$ does not act transitively on $S^{\ell_+}$.  This leaves the following possibilities:  row $1$ with $m=4$, row $4$, row $5$ with $m=4$, row $6$, row $8$ with $m=4$, and row $9$ for all $n\geq 4$; these are all tabulated in Table \ref{table:corank2part1}.

\begin{table}

\begin{tabular}{c|c|c|c|c|c}

$G$ & $K^+$ & $\ell_-$ & $\ell_- + \ell_+$ & $\ell_+$ & Notes\\

\hline

$SU(4)$ & $SU(2)$ & $5$ & $7$ & $2$ & multiple\\

%$SU(6)$ & $SO(6)$ & $9$ & $11$ & $2$ & \\

%$SU(6)$ & $Sp(3)$ & $5$ & $9$ & $4$ &\\

$SU(5)$ & $Sp(2)$ & $5$ & $9$ & $4$ & multiple\\

%$SU(5)$ & $SO(5)$ & $5$ & $9$ & $4$ &\\

%$SU(4)$ & $SU(2)$ & $5$ & $7$ & $2$ & Irreducible $4$-dim rep\\

$Spin(9)$ & $Spin(5)$ & $11$ & $15$ & $4$ &\\

$Spin(9)$ & $Sp(2)$ & $11$ & $15$ & $4$ \\

%$Spin(9)$ & $\mathbf{G}_2$ & $7$ & $15$ & $8$ \\

%$Spin(7)$ & $SU(2)$ & $7$ & $11$ & $4$ & multiple\\

$Sp(4)$ & $Sp(2)$ & $11$ & $15$ & $4$ & \\

%$Sp(3)$ & $Sp(1)$ & $7$ & $11$& $4$ & Several embeddings \\

$Spin(2m)$ & $Spin(2m-3)$ & $2m-1$ & $4m-5$ & $2m-4$ & $m\geq 4$\\

%$Spin(8)$ & $\mathbf{G}_2$ & $7$ & $7$ & $0$\\

%$\mathbf{E}_6$ & $\mathbf{F}_4$ & $9$ & $17$ & $8$\\

%$\mathbf{F}_4$ & $\mathbf{G}_2$ & $15$ & $23$ & $8$\\

%$\mathbf{G}_2$ & $\{e\}$ & $3$ & $11$ & $8$ & 

\end{tabular}

\caption{Subset of Table \ref{table:corank2} with $K^+ = L$ acting transitively on $S^{\ell_+}$}
\label{table:corank2part1}
\end{table}

\begin{proposition}  Except for row $2$ and the $n=5$ case of row $6$ in Table \ref{table:corank2part1}, the other possibilities do not correspond to cohomogeneity one actions on rational spheres.
\end{proposition}

\begin{proof}

We show that in all these cases, there is no $H\subseteq K^+\cap K^-$ with $K^\pm/H\cong S^{\ell_\pm}$.

For row 1,3,4, and $m\geq 6$ in row $6$ of Table \ref{table:corank2part1} we argue as follows.  By inspection, $K^+/H \cong S^{\ell_+}$ implies that $H$ must be, up to cover, isomorphic to either $S^1$, or $Sp(1)^2$, or $Spin(2m-4)$.  Thus, in all of theses cases, from Table \ref{table:transsphere}, $H$ does not contain the isotropy group of any transitive action on $S^{\ell_-}$.

For the $m=4$ case of row $6$, we find that up to cover, $G = Spin(8)$, $K^+ = Sp(2)$, and $H = Spin(4)$ which then implies that $K^-= Sp(2)\times Sp(1)$ and where the embedding of $H$ into $K^-$ must project onto the $Sp(1)$ factor.  This exact setup is found in \cite[7.8]{Asoh81}, where he shows the $8$-dimensional representation of $H$ given by $H\subseteq K^-\subseteq G\subseteq SU(8)$ has a different character from the representation of $H$ given by any $H\subseteq K^+\subseteq G\subseteq SU(8)$.  In particular, there is no copy of $H=Spin(4)$ which lies in $K^+\cap K^-$ and projects onto the $Sp(1)$ factor of $K^-$, so this possibility is ruled out as well. 

\end{proof}

We conclude the proof of Theorem \ref{thm:gsimp} by dealing with the remaining two possibilities from Table \ref{table:corank2part1}: row $2$ and the $n=5$ case of row $6$.

\begin{proposition}  For both row $2$ and the $n=5$ case of row $6$ of Table \ref{table:corank2part1}, there is, up to equivalence, a unique each cohomogeneity one action of $G$ on a rational sphere $M^n$ with specified $K^+$ (up to cover), and $\ell_\pm$.  When $G = SU(5)$, the action on $M\cong S^{19}$ is via the representation $\Lambda^2 \mathbb{C}^5$ of $SU(5)$, and when $G = Spin(10)$, $M\cong S^{31}$ and the action is via the spin representation.

\end{proposition}

\begin{proof}Asoh \cite[7.7 and 7.9]{Asoh81} proves that for both possibilities, the data of $G$ and $\ell_\pm$ determines the group diagram $H\subseteq K^\pm \subseteq G$ up to equivalence.  So, there is, up to equivalence, at most one cohomogeneity one action on a rational sphere $M^n$.

On the other hand, Straume \cite[Table II: 8b, 9b]{Straume96} indicates that the $SU(5)$ representation $\Lambda^2 \mathbb{C}^5$ is cohomogeneity one on $S^{19}$, and that the spin representation of $Spin(10)$ on $S^{31}$ is also cohomogeneity one.  In addition, he finds that $H = SU(2)^2 = Spin(4)$ for $G = SU(5)$, and that $H = SU(4)=Spin(6)$ for $G = Spin(10)$.   Straume does not list the $\ell_\pm$, but one can deduce them as follows.  Since $H$ is connected, both singular orbits must be orientable, so these examples must occur in Case(4), (5), or (6) of Theorem \ref{thm:GH}.  Moreover, in Case (5), we showed that $K^+ = K^- = G$.  However, $G/H$ is not a sphere for either pair $(G,H)$.  Also, we have already classified the possibilities for G in Case (6) (Proposition \ref{prop:gtype}), so this case is ruled out as well.  Thus, these two representations certainly fall with Case (4) of Theorem \ref{thm:GH}.

If $\ell_-$ and $\ell_+$ have the same parity, Proposition \ref{prop:Connecting} implies $\ell_- + \ell_+ = n-1$, while $2(\ell_- + \ell_+) = n-1$ holds if $\ell_-$ and $\ell_+$ have different parities.  When $H = SU(2)^2$, Table \ref{table:transsphere} shows that $\ell_{\pm}\in\{1,3,4,5\}$, while when $H = Spin(6)$, we find $\ell_{\pm}\in\{6,9\}$.  It is now obvious that we must have $(\ell_-,\ell_+) = ( 5,4), (9,6)$, for $H = SU(2)^2, SU(4)$ respectively.

\end{proof}

\subsection{$G$ has a circle factor}

For the duration of this subsection, in addition to Convention \ref{convention}, we will always assume that $\ell_- = 1$, $\ell_+$ is even, and that $G = G_1\times S^1$.   The case $\ell_+ =2$ implies $n = 7$, so $M$ is classified in Section \ref{sec:atmost7}.  Thus, we actually assume $\ell_+\geq 4$. We prove:

\begin{theorem}\label{thm:gcircle}Suppose $M^n$ is rational sphere on which a compact Lie group $G\cong G_1\times S^1 $ acts via a cohomogeneity one action.  Assume the homotopy fiber of the inclusion of a principal orbit $G/H$ into $M$ falls within Case (4) of Theorem \ref{thm:GH} with $\ell_+ \geq 4$.  Then, $M^n$ is orbit equivariantly diffeomorphic to a Brieskorn variety with standard action by $SO((n+1)/2)\times SO(2)$ or, when $n = 15$, by the restriction of this action to $Spin(7)\times SO(2)\subseteq SO(8)\times SO(2)$.

\end{theorem}

We begin with a classification of the relevant groups.

\begin{proposition}  Suppose $G = G_1\times S^1$ acts via an irreducible cohomogeneity one action on a rational sphere $M$.  Suppose the homotopy fiber of the inclusion of a principal orbit $G/H$ into $M$ belongs to Case (4) of Theorem \ref{thm:GH} with $\ell_- = 1$ and $\ell_+$ even.  Then $K^- = H\cdot S^1$ and $(G_1, (K^+)^0, H^0)$ is one of $(Spin(\ell_+ +2), Spin(\ell_++1), Spin(\ell_+))$ or $( Spin(7), \mathbf{G}_2, SU(3))$.

\end{proposition}

\begin{proof} The fact that $K^- = H\cdot S^1$ is from Proposition \ref{prop:ellinfo}.  As $K^-/H = S^1$, it follows that $H^0$ acts trivially on $K^-/H$.  Because the action must be primitive (Theorem \ref{thm:primitive}), Proposition \ref{prop:finiteintersection} implies that no simple factor of $H^0$ can act trivially on $S^{\ell_+} = K^+/H$.  From the classification of transitive actions on spheres in Table \ref{table:transsphere}, it follows that up to cover $( (K^+)^0, H^0)$ is either $(Spin(\ell_+ +1), Spin(\ell_+))$ or $(\mathbf{G}_2, SU(3))$.

Irreducibility implies $H^0\subseteq G_1$.  Because $H\subseteq K^+$ has full rank, this implies that $(K^+)^0\subseteq G_1$ as well.  From the covering $G_1/(K^+)^0\times S^1\rightarrow G/K^+$, it easily follows from Proposition \ref{prop:lasttop} that $G_1/(K^+)^0$ is a rational $S^{\ell_+ + 1}$ sphere.  From Table \ref{table:transsphere}, we see that for $n\geq 3$, if $(K^+)^0$ is up to cover  $Spin(\ell_++1)$, then $G = Spin(\ell_++2)$.  Further, it follows from this that $(K^0) = Spin(\ell_++1)$ and $H^0 = Spin(\ell_+)$, without the need to pass to covers.

When $n = 2$ or $n=1$, the exceptional isomorphisms $Spin(5)\cong Sp(2)$ and $Spin(3)\cong SU(2)$ potentially give rise to the possibilities $G_1 = Sp(3)$ and $G_1 = SU(3)$ respectively.  However, since $G_1/(K^+)^0$ is a rational $S^{\ell_+  +1}$ with $\ell_+ = 4$ and $2$ respectively, we contradict the fact that $Sp(3)/Sp(2) = S^{11}$ and $SU(3)/SU(2) = S^5$.

Lastly, since $\mathbf{G}_2$ does not cover any Lie group, in the last case we have $(K^+)^0 = \mathbf{G}_2$, $H^0 = SU(3)$, and from the classification of transitive actions on rational spheres, that $G_1= Spin(7)$.

\end{proof}

For $\ell_+\neq 6$, the fact that $G_1/(K^+)^0\cong S^{\ell_+ + 1}$ and $K^+/H\cong S^{\ell_+}$ determines the embedding of $H^0, (K^+)^0\subseteq G$ up to conjugacy.  When $\ell_+ = 6$, there are multiple conjugacy classes of embeddings of $Spin(7)$ into $Spin(8)$, but the outer automorphism group of $Spin(8)$ acts transitively on them.  Likewise, the embedding of $\mathbf{G}_2$ in $Spin(7)$ is unique up to conjugacy, as is the embedding of $SU(3)$ into $\mathbf{G}_2$.  Thus, we may always assume that $(K^+)^0$, $H^0$, and hence the semi-simple part of $K^-$ are embedded into $G_1$ in the standard way.

In fact, since $H^0$ is the lift of the usual block embedding $SO(\ell_+)\subseteq SO(\ell_+ + 2)$, $H_0$ contains the kernel of the projection $Spin(\ell_+)\rightarrow SO(\ell_+)$.  As such, we can and will replace the triple $(Spin(\ell_+ + 2), Spin(\ell_+ 1), Spin(\ell_+))$ with $(SO(\ell_++2), SO(\ell_+ + 1), SO(\ell_+))$.  For concreteness, we take $SO(\ell_+ +1)$ to fix the basis vector $e_1\in \mathbb{R}^{\ell_+ + 2}$, and we take $SO(\ell_+)$ to fix $e_1$ and $e_2$.

We note that the centralizer of $SO(\ell_+)$ in $SO(\ell_+ + 2)$ consists of rotations in the $e_1,e_2$-plane.  For the triple $(Spin(7), \mathbf{G}_2, SU(3))$ we note the centralizer of $SU(3)$ in $Spin(7)$ is also a circle, lifting of the center of $U(3)\subseteq SO(7)$.  In both cases, $G_1/H_0\cong T^1 S^{\ell_+ + 1}$ which $Z_{G_1}(H_0)$ acts on by rotation of the $e_1,e_2$ plane.  We denote this rotation by $R(\theta)$; note that this action is effective.  Then $G/K^- \cong (G_1/H_0 \times S^1)/S^1$ where $z=e^{i\theta}$ in the $S^1$ factor of $K^-$ acts via $R(a\theta)$ on the first factor, and $z^b$ on the second.  Of course, we may assume $\gcd(a,b) = 1$.

However, as $G/K^-$ is simply connected by Proposition \ref{prop:ellinfo},  we must, in fact, have $b= 1$.  Indeed, we obviously cannot have $b = 0$, and if $b\geq 1$, then from \cite[Lemma 1.3]{KapovitchZiller04}, $G/K^-$ is diffeomorphic to $T^1 S^{\ell_+ + 1}/ (\mathbb{Z}/b\mathbb{Z}) $, the fundamental group of $G/K^-$ has $b$.  Hence, $b=1$.  Note then that the projection map $p_2:K^-\rightarrow S^1$ has kernel $H^0$.

\begin{proposition}  Assume $K^+$ is connected.  If $a=0$, the corresponding action is not on a rational sphere, while if $a\neq 0$, the group diagram corresponds to the usual action on the Brieskorn variety $M^{2\ell_+ + 3}_d$ with $d$ even, $a = d/2$.

\end{proposition}

\begin{proof}  If $a=0$, then the action is not primitive, with both $K^\pm$ contained in $K^+\times S^1\subseteq G_1\times S^1$.  By Theorem \ref{thm:primitive}, $M$ cannot be in a rational sphere in this case.  On the other hand, if $a\neq 0$ and $K^+$ is connected, the resulting group diagram is the effective version of the Brieskorn action on $M_d^{2\ell_+ + 3}$ with $d$ even, $a = d/2$, found in \cite{GroveVerdianiWilkingZiller06}.

\end{proof}

Hence, for the remainder of this section, we may assume that $K^+$ is disconnected.  Since $K^+/H\cong S^{\ell_+}$ with $\ell_+\geq 2$, $H$ has the same number of components as $K^+$.  In particular, $H$ is also disconnected.  Note that if $H$ contains a non-identity element of the form $(1,z)\in G_1\times S^1$, this element acts trivially on $M$.  By passing to the quotient of $G/\langle (1,z)\rangle\cong G$, we may always assume $H$ does not contain any elements of this form.  Then \cite[7.10,7.12]{Asoh81} proves that $a$ must be odd, and that the group diagram is completely determined by $a$.

When $(G_1,(K^+)^0,H_0) = (SO(\ell_+  +2), SO(\ell_+ +1), SO(\ell_+))$, the diagram Asoh gets is precisely the group diagram for the usual $SO(\ell_+ + 2)\times SO(2)$ action on the Brieskorn variety $M^{2\ell_ + + 3}_a$ for$a$ odd\cite{GroveVerdianiWilkingZiller06}.  In the special case where $\ell_+ = 6$, the restriction of this action to $Spin(7)\times SO(2)$ is also cohomogeneity one with precisely the same group diagram as described above \cite{Straume96}.  Thus, we have completed the proof of Theorem \ref{thm:gcircle}.

\subsection{$G$ is semi-simple but not simple}

In this section, in addition to Convention \ref{convention}, we assume $G$ is not simple.  If $G = G_1\times ... \times G_k$, then we will use the notation $p_i:G\rightarrow G_i$ (for $i=1,..,k$) for the projection to the $i$th factor, and we will similarly use the notation $p_{12}:G\rightarrow G_1\times G_2$ for the projection onto the first two factors, etc.

From Proposition \ref{prop:gstructure}, if $G$ is semi-simple but not simple, then we must have $\ell_- > 1$.  In particular, all groups $H,K^\pm, G$ are connected from Proposition \ref{prop:ellinfo}.   As $\ell_-$ and $\ell_+$ have different parities, we assume for the rest of this section that $\ell_-$ is odd and $\ell_+$ is even.  The main goal of this section is to prove the following:

\begin{theorem}\label{thm:difparity}  Suppose $M$ is a rational sphere and $G$ is a non-simple semi-simple compact Lie group acting via a cohomogeneity one action on $M$.  Assume the inclusion of a principal orbit $G/H\rightarrow M$ has homotopy fiber occurring Case (4) of Theorem \ref{thm:GH}, with $\ell_-\geq 3$ odd and $\ell_+$ even.  If the action is not orbit equivalent to a linear action on a standard sphere, then $(\ell_-,\ell_+) = (3,2)$, and the action is orbit equivalent to an action by $SU(3)\times SU(2)$ on one of two rational $11$-spheres of the form $\mathbf{G}_2/SU(2)$.

\end{theorem}

Suppose $L\subseteq G_1\times G_2$.  We say a subgroup $L'\subseteq L$ is \textit{diagonally embedded} if $L'$ is a normal connected subgroup of positive dimension and both projections $L'\rightarrow G_i$ for $i=1,2$ have finite kernel.

We establish the following notation:  If $L\subseteq G_1\times ...\times G_k$, we write $L= L_1\cdot ...\cdot L_k\cdot L'$ to mean $L_1\times ...\times L_k\times L'$ is a finite cover of $L$ for which the image of $L_1$ in $L$ is equal to the kernel of the projection map $p_{2,...,k}|_L$, with the corresponding definition for each $L_i$, and with the image of $L'$ in $L$ diagonally embedded.  Occasionally, we will refer to bundles of the form, e.g., $L'\rightarrow (G_1/L_1\times ... \times G_k/K_k)\rightarrow G/L$.  Strictly speaking, the fiber is finitely covered by $L'$, but is not typically diffeomorphic to $L'$.  As the arguments will only really need the rational homotopy groups of $L'$, we will still write bundles in this form.

\begin{proposition}\label{prop:product}  If $L\subseteq G_1\times G_2$ is a connected Lie group which has no diagonally embedded subgroup, then $L$ is a product $L = L_1\times L_2$ with $L_i\subseteq G_i$.

\end{proposition}

\begin{proof}

Let $\mathfrak{l}\subseteq \mathfrak{g}_1\oplus \mathfrak{g}_2$ denote the Lie algebras of $L$ and $G_1\times G_2$.  Set $\mathfrak{l}_1 \subseteq \mathfrak{l}$ to be the kernel of the projection map $\mathfrak{l}\rightarrow \mathfrak{g}_2$ and likewise, set $\mathfrak{l}_2$ to be the kernel of the projection map to $\mathfrak{g}_1$.  Then obviously $\mathfrak{l}_1\oplus \mathfrak{l}_2 \subseteq \mathfrak{l}$.  Since $L$ has no diagonally embded subgroups, any ideal $\mathfrak{n}$ of $\mathfrak{l}$ must either lie in $\mathfrak{g}_1$ or in $\mathfrak{g}_2$.  Since $\mathfrak{l}$ is a sum of its ideals, $\mathfrak{l}\subseteq \mathfrak{l}_1\oplus\mathfrak{l}_2$.  Thus, $\mathfrak{l} = \mathfrak{l}_1\oplus \mathfrak{l}_2$, and then after exponentiating, we see $L = L_1\times L_2$ where $L_i$ is the unique connected subgroup of $G_i$ with Lie algebra $\mathfrak{l}_i$.

\end{proof}

\begin{proposition}\label{prop:diag}  Suppose $G$ is a compact connected Lie group, $G = G_1\times G_2$ with both $G_i$ non-trivial.  Suppose $G$ acts on a rational sphere $M$ via a cohomogeneity one action for which the homotopy fiber of the inclusion of a principal orbit falls within Case (4) of Theorem \ref{thm:GH}, and for which $\ell_-$ and $\ell_+$ have different parities.  Then at least one of $K^\pm$ must have a diagonally embedded factor.

\end{proposition}

\begin{proof}

Assume for a contradiction the proposition is false.  From Proposition \ref{prop:product}, we may write $K^- = K_1\times K_2$ with each $K_i\subseteq G_i$, and similarly $K^+ = L_1\times L_2$.  Since $H$ has full rank in $K^+$, $H = H_1\times H_2$ with $H_i\subseteq K_i\cap L_i$.  Further, since $S^{\ell_+} \cong K^+/H \cong L_1/H_1\times L_2/H_2$, it follows that up to swapping $G_1$ and $G_2$, that $H_1 = L_1$ and $L_2/H_2 = S^{\ell_+}$.

It follows that $H,K^\pm \subseteq K_1\times G_2$.  Thus, the action is not primitive, contradicting Theorem \ref{thm:primitive}, unless $K_1 = G_1$.  Now, $S^{\ell_-} \cong K^-/H\cong G_1/H_1 \times K_2/H_2$.  We note that $H_1\neq G_1$ because the action is irreducible, so we must have $H_2 = K_2$, so $S^{\ell_-}\cong G_1/H_1$.  It follows that $H,K^\pm\subseteq G_1\times L_2$, which again gives a non-primitive action unless $L_2 = G_2$.

But now we reach a contradiction:  $G/H = G_1/H_1\times G_2/H_2 \cong S^{\ell_-}\times S^{\ell_+}$, contradicting the rational cohomology calculation in Proposition \ref{prop:lasttop}.

\end{proof}

\begin{remark}  In the proposition above, the $G_i$ are not necessarily simple factors.  \end{remark}

From \cite[Theorem 4.8]{Totaro02}, as $\dim \pi_{odd}^\Q(G/H) = 3$ and since $H$ does not project onto any simple factor of $G$ (Convention \ref{convention}), it follows that $G$ has at most three simple factors.  Since $\dim\pi_{odd}^\Q(G/K^-) - \dim \pi_{even}^\Q(G/K^-) = 1$, $K^-$ has corank $1$ in $G$.

%A corank subgroup of a product of two groups falls into one of two types:  either it is a product $K_1\times K_2$ with one $K_i$ full rank in $G_i$ and the other of corank $1$, or it is of the form $(K_1\times K_2)\cdot K_{12}$ with both $K_i$ corank $1$ in $G_i$, and $K_{12}$, of rank $1$, diagonally embedded.  Further, if $K^-$ is a corank $1$ subgroup of $G_1\times G_2\times G_3$, then up to permuting the $G_i$, two possibilities occurs.  First, $K^-$ could be a product $K_1\times K_2\times K_3$ with $K_i\subseteq G_i$, where $K_1$ and $K_2$ have full rank in $G_1$ and $G_2$, and where $K_3$ has corank $1$ in $G_3$.  Second, it could be of the form $K_1\times (K_2\times K_3)\cdot K$ with two $K_2$ and $K_3$ corank $1$ in $G_2$ and $G_3$, $K_1$ full rank in $G_1$, and $K$ diagonally embeded into $G_2\times G_3$.

We may now deal with the case where $G$ has three simple factors.

\begin{proposition}\label{prop:2simp}  Suppose $G$ is a compact Lie group acting via a cohomogeneity one action on a rational sphere $M$.  Suppose the homotopy fiber of the inclusion of a principal leaf falls within Case (4) of Theorem \ref{thm:GH}, and that $\ell_-$ is odd and $\ell_+$ is even.  Then $G$ cannot be a product of three simple factors.

\end{proposition}

\begin{proof}We write $G = G_1\times G_2\times G_3$, $K^+ = L_1\cdot L_2\cdot L_3\cdot L'$, $K^- = K_1\cdot K_2\cdot K_3\cdot K'$, and $H = H_1\cdot H_2\cdot H_3\cdot H'$.

We first claim that $K^+$ must be diagonally embedded, i.e., that $L'\neq 0$.  To see this, assume it is false.  Then Proposition \ref{prop:diag} implies that $K'\neq 0$.  As $K^-$ has corank $1$ in $G$, it follows that $K'$ projects non-trivially to precisely two factors of $G$, say $G_2$ and $G_3$.  Then with respect to this decomposition of $G$ into two factors as $G_1\times (G_2\times G_3)$, $K^-$ has no diagonally embedded subgroup, contradicting Proposition \ref{prop:diag}.  Thus, $L'\neq 0$.

We next claim that for each $i=1,2,3$,  $L_i\subsetneq G_i$.  To see this, assume for a contradiction that it is false, so, without loss of generality, that $L_1 = G_1$.  Since $H\subseteq K^+$ has full rank, each $H_i\subseteq L_i$ and $H'\subseteq L'$ have full rank.  Now, $H_1\subsetneq L_1 = G_1$ since the action is irreducible.  As $K^+/H = S^{\ell_+}$ is, up to cover, diffeomorphic to $L_1/H_1 \times L_2/H_2\times L_3\times H_3\times L'/H'$ and $H_1\neq L_1$, it follows that $H_2 = L_2, H_3 = L_3, H' = L'$.  Since $H\subseteq K^-$, it now follows that $K^\pm\subseteq G_1\times p_{23}(K^-)$.  Since the action is primitive (Theorem \ref{thm:primitive}), we must have $p_{23}(K^-) = G_2\times G_3$.

Now, there is an obvious map $K^-/H\rightarrow p_{23}(K^-)/p_{23}(H)$ which is a fiber bundle with fiber $K_1\cdot H/H$.  As $K^-/H\cong S^{\ell_-}$, it follows from \cite[Corollary 2.5]{DeVitoKennard20} that $\dim \pi_{odd}(p_{23}(K^-)/p_{23}(H))) \leq 1$.  By \cite[Theorem 4.8]{Totaro02}, this can only happen if $H$ projects surjectively onto at least one of $G_2$ or $G_3$, contradicting irreducibility.  Thus, we conclude that each $L_i\subseteq G_i$ is proper.

\bigskip

Now, if $L'$ does not project surjectively onto some $G_i$, then again by Totaro's theorem, $\dim \pi_{odd}(G/K^-)\otimes \Q \geq 3$, contradicting Proposition \ref{prop:lasttop}.  Thus, $L'$ projects surjectively onto, say, $G_3$, which implies $L_3 = H_3 = 0$.  Irreducibility implies $H'\subsetneq L'$, which then implies, as above, that $H_1 = L_1$ and $H_2 = L_2$.  Now, either $L'$ projects non-trivially to two factors of $G$, or it projects non-trivially to all three factors of $G$.  We conclude the proof by handling each of these cases separately.

Assume initially that $L'$ projects non-trivially to exactly two factors, say, $G_2$ and $G_3$.  Interpreting $G$ as a product of two factors $G_1\times (G_2\times G_3)$ and using Proposition \ref{prop:diag}, it follows that $K'$ must have non-trivial projection to $G_1$.   Now, since $H_1 = L_1$ and $L'$ has trivial projection to $G_1$, it follows that $K^\pm \subseteq (K_1\cdot p_1(K'))\times G_2\times G_3$.  Primitivity forces $K_1\cdot p_1(K') = G_1$, and, as $p_1(K')$ is non-trivial, it now follows that $K_1 = 0$, so $L_1 = H_1\subseteq K_1= 0$ as well.

Thus, $H = H_2\cdot H'$ with $H'$ diagonally embedded in $G_2$ and $G_3$, while $K^- = K_2\cdot K_3\cdot K'$ with $K'$ diagonally embedded in $G_1$ and some other $G_i$.  As $H$ projects trivially to $G_1$, it must also project trivially to $K'$, so $H\subseteq K_2\times K_3$.  Thus $K^-/H\cong S^{\ell_-} \cong  ((K_2\times K_3)/H)\cdot K'$.  Since $K'$ is non-trivial, $K = SU(2)$ and $H = K_2\times K_3$.  But this contradicts the fact that $H'$ is diagonally embedded in $G_2\times G_3$.  This concludes the case where $L'$ projects to exactly two factors of $G$.

Thus, the only remaining case is when $L'$ projects non-trivally to all three $G_i$.  It follows that $L'$ has rank $1$ and, since $K^+$ is semi-simple, that $L' = SU(2)$ up to cover.  This is implies that $G_3 = SU(2)$ as well.  Then $G/K^+\cong G_1/L_1\times G_2/L_2$.  Since both $L_i\subsetneq G_i$ and $H^\ast(G/K^+;\Q)\cong H^\ast(S^{\ell_-}\times S^{\ell_+};\Q)$, it follows that both $G_i/L_i$ are rational spheres, of dimension $\ell_-$ and $\ell_- + \ell_+$.   Since $L'= SU(2)$ projects non-trivally to both $G_1$ and $G_2$, it follows that $L_i\subseteq G_i$ extends to $L_i\cdot SU(2)\subseteq G_i$, which implies that both $\ell_-$ and $\ell_- + \ell_+$ are congruent to $3$ mod $4$.  Thus, $\ell_+$ is divisible by $4$.  On the other hand,  if $L' = SU(2)$, then $H' = S^1$.  From Proposition \ref{prop:gstructure}, it follows that $\ell_+ = 2$, giving the final contradiction.

\end{proof}

We henceforth assume $G = G_1\times G_2$ has precisely two simple factors.  We will write $K^- = (K_1\times K_2)\cdot K'$, $K^+ = (L_1\times L_2)\cdot L'$, and $H = (H_1\times H_2)\cdot H'$, with the primed groups diagonally embedded.

\begin{proposition}\label{prop:kernel}Under the assumptions of Theorem \ref{thm:difparity} with $\ell_+$ even and $\ell_-$ odd, at least one of the two projections $p_i:K^+\rightarrow G_i$ has finite kernel.

\end{proposition}

\begin{proof}  Write $G = G_1\times G_2$, $K^+ = (L_1\times L_2)\cdot L'$, $K^- = (K_1\times K_2)\cdot K'$, and $H = (H_1\times H_2)\cdot H'$.  We must prove that at least one $L_i$ is trivial.

So, assume for a contradiction that both $L_1$ and $L_2$ are non-trivial.  Since $K^+$ is semi-simple from Proposition \ref{prop:gstructure}, it follows that $\pi^\Q_3(G_i/L_i) =0$.  From the long exact sequence in rational homotopy groups associated to the bundle $L'\rightarrow G_1/L_1\times G_2/L_2\rightarrow G/K^+$, we deduce that if $L'\neq 0$, then $\pi^\Q_4 (G/K^+)\neq 0$, contradicting Proposition \ref{prop:lasttop}.  Thus, $L' = 0$, so by Proposition \ref{prop:product}, we may write $K^+ = L_1\times L_2\subseteq G_1\times G_2$.

Since $H$ has full rank in $K^+$, it now follows that $H' = 0$, and $S^{\ell_+}\cong (L_1/H_1) \times (L_2/H_2)$.  Thus, we may assume without loss of generality that $H_1 = L_1$.

Since $L'= 0$, Proposition \ref{prop:diag} implies $K'\neq 0$.  Since $H'=0$ and $H\subseteq K^-$, $H\subseteq K_1\times K_2$.  Then $S^{\ell_-}\cong K^-/H\cong ((K_1\times K_2)/H)\cdot K' $.  It follows that $H = K_1\times K_2$.  In particular $H_1$ acts trivially on both $S^{\ell_\pm}$, contradicting Proposition \ref{prop:finiteintersection}.

\end{proof}

We hencefoth assume that $L_2 = 0$.  This implies that $H_2 = 0$ as well, so $K^+ = L_1\cdot L'$ and $H  = H_1\cdot H'$.  We now break into further subsections, depending on whether or nor $(\ell_-, \ell_+) = (3,2)$ or not.

\subsubsection{The case where $(\ell_-, \ell_+) = (3,2)$}\label{sec:32}

The main goal of this subsection is to classify all actions on rational spheres of dimension $11$  for which $G$ is a product of two simple factors and $(\ell_-, \ell_+) = (3,2)$.  In order to state the main result, recall that $\mathbf{G}_2$ contains subgroups (unique up to conjugacy) which are isomorphic to $SU(3)$ and $SO(4)$.   Writing the double cover of $SO(4)$ as $SU(2)_1\times SU(2)_3$ (with the subscript refering to the Dynkin index of $SU(2)_i\subseteq \mathbf{G}_2$), we have a natural action of $SU(3)\times SU(2)_i$ on $\mathbf{G}_2/SU(2)_{4-i}$ given by $(A,B)\ast C\, SU(2)_{4-i} = ACB^{-1}\, SU(2)_{4-i}$.  Since the two $SU(2)$s in $SO(4)$ centralize each other, this action is well defined.

\begin{theorem}\label{thm:23}  Suppose $G = G_1\times G_2$ is a product of two non-trivial compact simple simply connected Lie groups and that $G$ acts on a rational sphere $M^{11}$ with  homotopy fiber $\mathcal{F}$ of the inclusion $G/H\rightarrow M$ of a principal orbit falling into Case (4) of Theorem \ref{thm:GH} with $\ell_- =2$ and $\ell_+ = 3$.  Then $G = SU(3)\times SU(2)$ and, up to equivalence, the action is one of three possibilities: the tensor product action of $SU(3)\times SU(2)$ on $S^{11}\subseteq \mathbb{C}^3\otimes\mathbb{C}^2$, or the above $SU(3)\times SU(2)_i$ action on $\mathbf{G}_2/SU(2)_{4-i}$ for $i=1,2$.  Conversely, these three actions are all cohomogeneity one actions meeting all the hypothesis in the first sentence.

\end{theorem}

We will prove this by a series of proposition.  We first classifying the groups which can appear in a group diagram of an action as in Theorem \ref{thm:23}, finding that all the groups are determined up to cover.   We then study the problem of the embeddings of $H,K^\pm$ into $G$.  We will find that, up to equivalence, there are precisely five cohomogeneity one actions which need to be considered.  We finally show that two of them correspond to actions on spaces which are not rational spheres, while the remaining three correspond to the possibilities listed in Theorem \ref{thm:23}.

\begin{proposition}  Suppose a compact Lie group $G = G_1\times G_2$, with both $G_i$ simple, acts on a rational sphere $M^{11}$ with homotopy fiber $\mathcal{F}\sim_\Q S^{2}\times S^{3}\times \Omega S^{6}$.  Then $G\cong SU(3)\times SU(2)$, $K^+\cong SU(2)$, and up to cover $K^- = S^1\times SU(2)$.  Further, up to conjugacy, $K^+$ is embedded into $G$ as either $\Delta SU(2)\subseteq SU(2)\times SU(2)$ or $\Delta SU(2)\subseteq SO(3)\times SU(2)$, and $K^-$ is embedded up to conjugacy as either $(z,A)\mapsto (\diag(zA,\overline{z}^{2}, A)$ or $A\mapsto (\diag(z^a A, \overline{z}^2, \diag(z^b, \overline{z}^b)$ for relatively prime integers $a$ and $b$.

\end{proposition}

\begin{proof}Since $\pi^\Q_3(G/H)\cong \pi^\Q_3(S^2\times S^3\times S^5)$ has dimension $2$ and $G = G_1\times G_2$ with both $G_i$ simple, it follows that $H$ must be a torus.  Since $\pi_2(G/H)\otimes \Q \cong \Q$, $H = S^1$.  Since $G/H$ has dimension $10$, we conclude from the classification of simple groups that $G = SU(3)\times SU(2)$, from which it easily follows that, up to cover, $K^+$ is isomorphic to $SU(2)$ and $K^-$ is isomorphic to $S^1\times SU(2)$.

We now focus on the embeddings.  Since $K^+ = L_1\cdot L'\cong SU(2)$, $p_1(K^+)$ is non-trivial.  Recalling that there, up to conjugacy, precisely two non-trivial maps from $SU(2)$ to $SU(3)$, being the block embedding, and the double cover $SU(2)\rightarrow SO(3)\subseteq SU(3)$, it follows that $p_1(K^+)$ is, up to conjugacy,  one of these two options.  As $p_1(H)\subseteq p_1(K^+)$, it now follows that $p_1(H)$ is conjugate to a subgroup of the form $\diag(z,z^{-1},1)$ for $z\in S^1$.

We next claim that $p_2(K^+)$ is also non-trivial.  For it is trivial, then $H' = L' = 0$, from which it follows that $H\subseteq K_1$.  Since $K^-/H$ is an odd dimensional sphere, we conclude that $K_1 = S^1$ and $K' =SU(2)$.  This must centralize $p_1(\{1\}\times SU(2)))$, so it follows that $K_1$ must be conjugate to $\diag(z,z,z^{-2})$.  This contradicts the fact that $K_1 = H$ is conjugate to $\diag(z,z^{-1},1).$  It now follows that $K^+$ is isomorphic to $SU(2)$.  In summary, up to conjugacy, we must have $K^+ = \Delta SU(2)\subseteq SU(2)\times SU(2)\subseteq G$ or $K^+ = \Delta SU(2)\subseteq SO(3)\times SU(2)\subseteq G$.

\bigskip

We now work towards obtaining the embedding of $K^-$.  We first claim that $p_1$ is non-trivial on the $SU(2)$ factor of $K^-$.  For if this is not true, then $p_1(K^-)\subseteq p_1(H)\subseteq p_1(K^+)$.  This follows as the projection of $H$ to the circle factor of $K^-$ is surjective, for otherwise $\pi_1(K^-/H)$ would be infinite.  Thus, if $SU(2)\subseteq K^-$ is in the kernel of $p_1$, the action is non-primitive, being contained in $p_1(K^+)\times SU(2)\subseteq G$.

Moreover, in fact, $p_1$ must be injective on the $SU(2)$ factor of $K^-$.  For otherwise, since $SO(3)\subseteq SU(3)$ is maximal, it would follow that $K^-$ is a product $SO(3)\times S^1$ in $SU(3)\times SU(2)$.  But then the long exact sequence in homotopy groups shows that the quotient $K^-/H \cong S^3$ is not simply connected, giving a contradiction.  Thus, the $SU(2)$ factor in $K^-$ maps, up to conjugacy, as the usual $SU(2)$ in $SU(3)$.  Further, since the $S^1$ factor of $K^-$ centralizes it, $p_1(K^-)$ must have the form $(\diag(z^a A, z^{-2a})$ for $(z,A)\in S^1\times SU(2)$.

Lastly we note that since $H\subseteq K^-$ and $p_2(H)$ is non-trivial, it follows that $p_2(K^-)$ is non-trivial.  If $p_2$ is non-trivial on the $SU(2)$ factor, then the $S^1$ factor must only embedded into $SU(3)$, so $a\neq 0$.  But then $K^-\subseteq G$ is independent of $a$, so without loss of generality, we may assume $a=1$.  On the other hand, if $p_2$ is non-trivial on the $S^1$ factor, we get the second listed form.

\end{proof}

We now fully classify the possible group diagrams.  Note that changing $(a,b)$ to $(-a,-b)$ does not change the image of $K^-$ and there is an inner automorphism of $G$ which changes $b$ to $-b$.  Hence, we may assume that both $a,b\geq 0$. 

\begin{proposition}\label{prop:23possibilities}  Suppose $G $, $K^\pm$ are as in the previous proposition.  There, up to equivalence, there are precisely five cohomogeneity one diagrams, and all five are determined by the conjugacy class of the embeddings of $H,K^+$ and $K^-$.  These embeddings are given in Table \ref{table:5examples}, where $A,B\in SU(2)$, $z\in S^1$, and $\pi:SU(2)\rightarrow SO(3)$ the usual double cover.

%If $K^-$ is embedded into $G$ as $(z,B)\mapsto (\diag(z^{-2}, zB), B)$, then $K^+$ is, up to conjugacy, embedded as $(A', A)$.

%If $K^-$ is embedded into $G$ as $(z,B)\mapsto (\diag(z^{-2a}, z^a B), \diag(z^b,\overline{z}^b)$, then $(a,b,K^+)$ are given in the following table.

\begin{center}
\begin{tabular}{c|c|c}

$H$ & $K^-$ & $K^+$\\

\hline

$(\diag(\overline{z}^2, z^2,1), \diag(z,\overline{z}))$ & $(\diag(z^{-2},zB), B)$ & $(\pi(A),A)$\\

$(\diag(\overline{z}^2,z^2,1), \diag(z,\overline{z}))$ & $(\diag(B,1) , \diag(z,\overline{z}))$ & $(\pi(A),A)$\\

$(\diag(\overline{z}, z,1), \diag(z,\overline{z}) )$ & $(\diag(B,1), \diag(z,\overline{z}))$ & $(\diag(A,1),A)$\\

$(\diag(\overline{z},z,1), \diag(z,\overline{z}))$ & $(\diag(z^{-2}, zB), \diag(z^2,\overline{z}^2))$ & $(\diag(A,1), A)$\\

$(\diag(\overline{z}^2, z^2, 1), \diag(z,\overline{z}))$ & $(\diag(z^{-2}, zB), \diag(z, \overline{z}))$ & $(\pi(A),A)$\\

\end{tabular}\label{table:5examples}

\end{center}

\end{proposition}

\begin{proof}%Recall that $H\cong S^1\subseteq K^-$ with $K^-/H\cong S^3$ and that, up to cover, $K^- = S^1\times SU(2)$.  Lifting $H$ to this cover, we obtain $\overline{H}\cong S^1\subseteq S^1\times SU(2)$.  Since $S^1\times SU(2)/H$ covers $K^-/H$, we must have $S^1\times SU(2)/H\cong S^3$.  Up to conjugacy, all $S^1\subseteq S^1\times SU(2)$ are described by a pair of relatively prime integers $(c,d)$, $S^1 = \{(z^c, \diag(z^d, \overline{z}^d))\}$.  If $c=0$, then $S^1\times SU(2)/\overline{H}\cong S^1\times S^2$ while if $c\neq 0$, then $S^1\times SU(2)/\overline{H}\cong SU(2)/(\mathbb{Z}/c\mathbb{Z})$.  Thus, we conclude $|c| = 1$.  By reparamaterizing $\overline{H}$ if necessary, we can assume $c=1$.  Likewise, by conjugating  by an appropriate element of $K^-$, we may assume $d \geq 0$.

%Next, assume $K^- = (\diag z^{-2a}, z^a B), \diag(z^b, z^{-b})$.  In this case, the image of $\overline{H}$ in $G$ is $(\diag(z^{-2a}, z^{a+d}, z^{a-d}), \diag(z^b, z^{-b})$.  Since $H$ must be conjugate to $\diag(z,z^{-1},1)$, we must have either $a=0$ or $a =  d$ (since both $a$ and $d$ are positive).  Assume first that $a=0$, which implies $b =  1$.  Since $p_1:K^+\rightarrow SU(3)$ is either injective or has kernel of order $2$, it now follows that $d =1$ or $2$.  If $d=1$, then up to conjugating by $H$, $K^+ = (\diag(1,A), A)$, and when $d=2$, we likewise conclude that $K^+ = (A',A)$.

%Lastly, assume $a=d$, so $H$ is the collection of all elements of the form $(\diag(z^{-2a}, z^{2a}, 1), \diag(z^b, z^{-b}))$ for $z\in S^1$.  Since $p_2:H\rightarrow SU(2)$ is injective, it follows that any $b$th root of unity must also be a $2a$th root of unity.  That is, $b|2a$.  On the other hand, since $a$ and $b$ are relatively prime, it now follows that $d=a=1$ and $b=1 $ or $b=2$.  As in the previous cases, we may finally conjugate by an element of $H$ to get $K^+$ in the forms given in the proposition.

We first prove that that diagram is determined up to equivalence solely by the conjugacy classes of $H$ in $K^-$ and both $K^\pm$ in $G$.  From Proposition \ref{prop:equiv}, we may assume the embedding of $H $and $K^-$ are fixed.   Suppose $K^+\subseteq G$, and $K\subseteq G$ is conjugate to $K^+$ and contains $H$.  We need to show the group diagram corresponding to $H\subseteq K^\pm \subseteq G$ is equivalent to that of $H\subseteq K^-, K\subseteq G$.  Because $K^+$ and $K$ are conjugate, and all copies of $H$ in $K^+$ are conjugate, it follows that $K = gK^+ g$ for some $g\in N_G(H)$. A simple computation reveals that $N_G(H) = T^3 \cup \sigma T^3$ where $T^3$ is a maximal torus of $G$ and $\sigma$ is the Weyl group element which acts as complex conjugation on $H\cong S^1$.  Since $K^+\cong SU(2)$, the Weyl group of $H\subseteq K^+$ contains an element $\rho$ which acts as complex conjugation on $H$.  Then $\rho\in K^+\cap (N_G(H)\setminus N_G(H)^0)$.  It follows that the $N_G(H)$ orbit of $K^+$ (via the conjugation action) coincides with the $N_G(H)^0$ orbit of $K^+$.  Thus, we may assume $g\in N_G(H)^0$.  By Proposition \ref{prop:equiv}, the diagram involving $K^+$ is equivalent to the diagram using $K$.

Now, we show that the condition that $H\subseteq K^\pm \subseteq G$ forms a valid group diagram determines the conjugacy classes of $H$ in $K^-$ and $K^\pm$ in $G$ up to precisely five possibilities.  Note first that elementary representation theory shows that all diagonal embeddings of $K^+$ into $G$ are, up to conjugacy, the two given ones.

Assume initially that $K^-$ is embedded into $G$ as $(\diag(\overline{z}^2, zB), B)$.   Because $H\subseteq K^+$, every element of $p_1(H)$ has $1$ as an eigenvalue.  We may conjugate $p_1(H)$  into the standard maximal torus of $K^-$,  elements of the form $(\diag(z^{-2}, zw, z\overline{w}), \diag(w,\overline{w}))$.  Then, we find that having an eigenvalue of $1$ forces either $z=1$, $zw=1$, or $z\overline{w} = 1$.  If the case $z=1$ occurs, then $H$ is contained in the semi-simple part of $K^-$, which contradicts that fact that $K^-/H\cong S^{\ell_-}$ with $\ell_- > 1$.  Thus, by replacing $w$ with $\overline{w}$ if necessary, we may assume $H$ has the form $(\diag(\overline{z}^2, z^2, 1), \diag(z,\overline{z}))$.  It follows that in this case, $K^+$ must be embedded as $\Delta SU(2)\subseteq SO(3)\times SU(2)$.

We may now assume $K^-$ is embedded into $G$ as $(\diag(\overline{z}^{2a}, z^a B), \diag(z^b, \overline{z}^b))$ with $\gcd(a,b) = 1$ and both $a,b\geq 0$.
We know that $H\subseteq K^+$ is embedded into $G$, up to conjugacy, either as $(\diag(1, z,\overline{z}), \diag(z,\overline{z}))$ or as $(\diag(1,z^2,\overline{z}^2), \diag(z,\overline{z})).$  

We seek conditions on $a,b$ which guarantee that $H$ is conjugate to a subgroup of $K^-$.  To that end, notice that every element in $K^-$ is conjugate to an element of the form $(\diag(z^{-2a}, z^a w ,z^a \overline{w}), \diag(z^b, \overline{z}^b))$.  By replacing $w$ with $\overline{w}$, we must have $2a=0$ or $\lambda = z^a$.  In the first case, $a = 0$, which then implies $b=1$.  In the second case, we find that $a=b$ or $2a=b$, depending on the embedding of $K^+$ into $G$.  Since $\gcd(a,b) = 1$, we must have $a=b=1$ or $2a=b=2$.

\end{proof}

We will now show that the two cases of Proposition \ref{prop:23possibilities} where the circle factor of $K^-$ is embedded only into the $SU(2)$ factor of $G$ correspond to actions on spaces which are not rational spheres.

\begin{proposition}  The actions corresponding to the second two entries of Table \ref{table:5examples} are not actions on rational spheres.

\end{proposition}

\begin{proof}

The action given in the third row is not primitive, with both $K^\pm\subseteq p_1(K^-)\times SU(2)$.  So, we focus on the case where $K^+$ is embedded as $K^+ = (\pi(A), A)\subseteq G$.  We will show the two induced maps $H^5(G/K^\pm;\Q)\rightarrow H^5(G/H;\Q)$ have the same image.  Assuming this for the moment, it then follows from the Mayer-Vietoris sequence for the double disk bundle decomposition of $M$ that $H^5(M;\Q)\neq 0$.

To compute the maps $H^\ast(G/K^\pm;\Q)\rightarrow H^\ast(G/H;\Q)$, we use the description of the rational cohomology of a homogeneous space as found in \cite[Theorem 2.71]{FelixOpreaTanre08}.  Let us first set up notation.

For a Lie group $L$, we use the notation $BL$ for the classifying space of $L$.  Then we may identify $$H^\ast(BH;\Q)\cong \Q[z_2], H^\ast(BK^+;\Q)\cong \Q[y_4], H^\ast(BK^-;\Q)\cong \Q[x_2, x_4], $$ and $$H^\ast(BG;\Q)\cong \Q[u_4,u_6,v_4]$$ with the subscripts denoting degrees of the elements.  Further, we let $sG$ denote $H^\ast(BG;\Q)$ with the degree of the generators all shifted down by one, so $sG\cong  \Lambda_\Q [su_4, su_6, sv_4]$ with $|su_i| = i-1$.

The inclusion $i:H\rightarrow G$ induces a map $Bi^\ast:H^\ast(BG;\Q)\rightarrow H^\ast(BH;\Q)$.  Then,\cite[Theorem 2.71]{FelixOpreaTanre08} gives an isomorphism $H^5(G/H;\Q)\cong H^5( \Q[z_2]\otimes sH, d)$ with $dz_2 = 0$ and $dsu_4 = Bi^\ast(u_4)$, $dsu_6 = Bi^\ast(u_6)$, and $dsv_4 = Bi^\ast(v_4)$.

A simple calculation shows that $Bi^\ast(su_4) = -4z_2^2$, $Bi^\ast(su_4) = -z_2^2$, and $Bi^\ast(su_6) = 0$.  Thus, $H^5(G/H;\Q)$ is isomorphic to $\Q^2$, generated by $[su_6]$ and $ [z_2(4sv_4 - su_4)]$.

In a similar fashion, the inclusions $i_\pm:K^\pm\rightarrow G$ induce the following:

$$\begin{matrix} Bi_+^\ast(su_4) = 4x_4 & Bi_+^\ast(su_6) = 0 & Bi_+^\ast(sv_4) = x_4\\
Bi_-^\ast(su_4) = y_4 & Bi_-^\ast(su_6) = 0 & Bi_-^\ast(sv_4) = y_2^2
\end{matrix}$$

It follows that in both cases $H^5(G/K^\pm;\Q)\cong Q$, generated by $[su_6]$.

Now, the inclusions $j_\pm:H\rightarrow K^\pm$ induce maps $H^\ast(BK^\pm;\Q)\otimes sH\rightarrow H^\ast(BH;\Q)\otimes sH$ of the form $Bj_\pm\otimes 1$ which, in turn, induce the obvious maps $H^\ast(G/K^\pm;\Q)\rightarrow H^\ast(G/H;\Q)$.  In particular, the image of both $H^\ast(G/K^\pm;\Q)$ in $H^\ast(G/H;\Q)$ is spanned by $[su_6]$.

\end{proof}

To complete the proof of Theorem \ref{thm:23}, we show that the remaining three actions as described in Proposition \ref{prop:23possibilities} correspond to three actions listed in Theorem \ref{thm:23}

\begin{proposition}  The first, fourth, and fifth entries of Table \ref{table:5examples} correspond to the three listed actions in Theorem \ref{thm:23}.

\end{proposition}

\begin{proof}When $K^+ = (\diag(1,A),A)$ and $K^- = (\diag(z^{-2}, zB), \diag(z,z^{-1}))$, this diagram is found in Asoh's work \cite[7.13]{Asoh81}, where he shows  this corresponds to the tensor product action of $SU(3)\times SU(2)$ on $\mathbb{C}^{3}\otimes \mathbb{C}^2$.  We now turn to the remaining two cases.

As shown by Iwata \cite{Iwata78}, the natural $SU(3)$ action on $\mathbf{G}_2/SO(4)$ is cohomogeneity one.  Writing the double cover of $SO(4)$ as $SU(2)_1\times SU(2)_3$, it follows immediately that the action by $SU(3)\times SU(2)_i$ on $\mathbf{G}_2/SU(2)_{4-i}$ given by $(A,B_i)\ast C SU(2)_{4-i} = ACB_i^{-1} SU(2)_{4-i}$ is a cohomogeneity one action.  Further, both $\mathbf{G}_2/SU(2)_i$ are rational $11$-spheres, as shown in, e.g., \cite{KapovitchZiller04}.  We need only to show that these two examples fall within Case (4) of Theorem \ref{thm:GH} with $(\ell_-,\ell_+) = (3,2)$.

By counting dimensions, the principal isotropy group of either of these two actions must be $H = S^1$ (up to components).  From Table \ref{table:transsphere}, the dimension of the spheres $K^\pm/H$ must be in $\{1,2,3\}$.  On the other hand, from Proposition \ref{prop:Connecting}, the loop space factor of the homotopy fiber of the inclusion $\mathcal{F}\rightarrow G/H\rightarrow \mathbf{G}_2/SU(2)_i$ must either be $\Omega S^6$ or $\Omega S^{11}$.  It now follows from Proposition \ref{prop:Connecting} that $\ell_+ = 2$ and $\ell_- = 3$ and that these examples must fall into Case (4) of Theorem \ref{thm:GH}.

\end{proof}

\begin{remark}  In fact, the first action in Table \ref{table:5examples} of Proposition \ref{prop:23possibilities}, with $K^- = (\diag(z^{-2}, zB), B)$, is the action of $SU(3)\times SU(2)_1$ on $\mathbf{G}_2/SU(2)_3$.  One can see using basic knowledge of $\mathbf{G}_2$ and the octonions $\mathbb{O}$ as follows.  Recall $\mathbf{G}_2$ acts transitively on $S^6\subseteq \operatorname{Im}(\mathbb{O})\cong \mathbb{R}^7$ with isotropy group at $i\in \mathbb{O}$ given by a copy of $SU(3)$.  This $SU(3)$ contains $SU(2)_1$.  Now, if $(A,B_1)\in SU(3)\times SU(2)_1$ is in the isotropy group at the identity coset $eSU(2)_3$, then there is a matrix $B_3\in SU(2)_3$ with $AB_1^{-1} = B_3$.  Since $SU(3)\subseteq \mathbf{G}_2$ is the isotropy at $i\in \mathbb{O}$, it follows that $B_3$ fixes $i$.  Since $B_3\in SU(2)_3$, this implies that $B_3\in Z_{SU(3)}(SU(2)_1)$.  Since $A = B_1 B_3^{-1}$, it is now straightforward to see that the isotropy group at $e$ is, up to conjugacy, $(\diag(z^{-2}, zB), B)$ with $B\in SU(2)$.  In a similar fashion, one can verify the other singular isotropy group occurs at $g SU(2)_3$ where $g\in \mathbf{G}_2$ moves a point in $\operatorname{Im}(\mathbb{H})^\bot\subseteq \operatorname{Im}(\mathbb{O})$ to $i$.  The identity and this $g$ are also singular points of the $SU(3)\times SU(2)_3$ action on $\mathbf{G}_2/SU(2)_1$.

\end{remark}

\subsubsection{The case where$(\ell_-, \ell_+)\neq (3,2)$}

In this section, we complete the proof of Theorem \ref{thm:paritycase} in the case where $G$ is semi-simple but not simple.  Recall that Propositions \ref{prop:2simp} and \ref{prop:kernel} imply that we may assume $G = G_1\times G_2$ is a product of two simple groups, and that $K^+ = L_1\cdot L'$, $H_1\cdot H'$, and $K^- = (K_1\times K_2)\cdot K'$, where the primed groups are diagonally embedded and, e.g., $K_1\subseteq G_1$ and $K_2\subseteq G_2$.  Moreover, from Proposition \ref{prop:diag}, at least one of $L'$ or $K'$ must be non-trivial.  Because the case of $(\ell_-,\ell_+) = (3,2)$ has already been handled in Section \ref{sec:32}, we assume $K^+/H \cong S^{\ell_+}$ and $K^-/H\cong S^{\ell_-}$ with $(\ell_-, \ell_+)\neq (3,2)$.  Recall that $\ell_-\geq 3$ is odd and $\ell_+\geq 2$ is even.  Finally, recall Convention \ref{convention}

We also note that from Proposition \ref{prop:lasttop}, it easily follows that $K^+$ has corank $2$ in $G$ and that $K^-$ has corank $1$ in $G$.  This easily implies that $\rk K'\leq 1$ and $\rk L'\leq 2$.

\begin{proposition}\label{prop:g2rank}  Under the above assumptions, $\rk L'\leq \rk G_2\leq 2$.
\end{proposition}

\begin{proof}
The inequality  $\rk L'\leq \rk G_2$ follows since the projection $p_2:G\rightarrow G_2$ has finite kernel when restricted to $L'$, so $\rk(L') = \rk(p_2(L')) \leq \rk G_2$.

For the other inequality, first observe that Proposition \ref{prop:lasttop} gives $H^\ast(G/K^+;\Q)\cong H^\ast(S^{\ell_- }\times S^{\ell_- + \ell_+};\Q)$, so $K^+$ has corank $2$ in $G$.  Since $K^+ = L_1\cdot L'$, $p_1:G\rightarrow G_1$ has finite kernel when restricts to $K^+\subseteq G$.  In particular, $\rk(K^+) = \rk(p_1(K^+)) \leq \rk G_1$.  Then $\rk G_1 + \rk G_2 - 2 = \rk K^+\leq \rk G_1$.  Thus, $\rk G_2\leq 2$.

\end{proof}

Of course, if $H_1$, and therefore $L_1$ are both trivial, the same proof shows that $\rk L'\leq \rk G_1\leq 2$.  We briefly handle this case in the following proposition.

\begin{proposition}\label{prop:h1trivial}  Suppose $H_1$ is trivial.  Then the quadruple $(H,K^-, K^+, G)$ is isomorphic to $(Sp(1)^2,Sp(1)^3,Sp(2),Sp(2)\times Sp(2))$.

\end{proposition}

We will later see in Proposition \ref{prop:spclass} that there is, up to equivalence, a unique cohomogeneity one action on a rational sphere having these isotropy groups.

\begin{proof}(Proof of Proposition \ref{prop:h1trivial})  As mentioned above, the hypotheses imply that both $G_1$ and $G_2$ have rank at most $2$.  As $\dim \pi^\Q_{odd}(G/H) = 3$, at least one $G_i$ has rank $2$.  Without loss of generality, we may assume $G_1$ has rank $2$.  We now break into cases depending on $G_2$.

If $G_2 = SU(2)$, the Proposition \ref{prop:g2rank} implies that $L'$ has rank $1$.  As $L'/H'\cong S^{\ell_+}$, we conclude that $L'\cong SU(2)$, $H'\cong S^1$, and that $\ell_+ = 2$.  Further, since $K^+$ is simple, it now follows that $\pi^\Q_3(G/K^+)\neq 0$.  As $\ell_-\geq 3$ by assumption, it follows that $\ell_- = 3$, contradicting the fact that $(\ell_-,\ell_+)\neq (3,2)$.

Thus, $G_2$ must have rank $2$.  Since $K^+$ has corank $2$, it must thus have rank $2$ and act transitively on $S^{\ell_+}$.  It follows from Table \ref{table:transsphere} that $\ell_+\in \{2,4,6\}$, corresponding to $K^+$ being, up to cover, isomorphic to $SU(2)\times SU(2), Sp(2),$ or $\mathbf{G}_2$.

If $\ell_+  = 2$, then, as $SU(2)\times SU(2)$ does not map to $SU(3)$ with finite kernel, both $G_i\in \{Sp(2), \mathbf{G}_2\}$.  In particular, the odd rational homotopy groups of $G$ occur in dimensions which are a subset of $\{3,7,11\}$.  Since $\pi^\Q_{odd}(G/K^+)$ is supported in dimensions $\ell_-$ and $\ell_- + \ell_+ = \ell_- + 2$, we have a contradiction, as no pair of elements in $\{3,7,11\}$ has a difference of $2$.

If $\ell_+ = 6$, then $K^+\cong \mathbf{G}_2$.  Since $\mathbf{G}_2$ has no non-trivial maps to any rank 2 Lie group other than itself, we must have $G_1 = G_2 = \mathbf{G}_2$.  In particular, $\dim \pi^\Q_{odd}(G/K^+)$ is supported in degrees $3$ and $11$, which have a difference of $8\neq 6$.  Thus, this case cannot occur either.

So, we must be in the case where $\ell_+ = 4$, so $K^+$ is covered by $Sp(2)$ and $H$ is covered by $Sp(1)^2$.  Since $\dim \pi^\Q_3(K^+)$, it follows that $\pi_3^\Q(G/K^+)$ is non-trivial.  Since $\ell_+ = 4$, $\pi^\Q_{odd}(G/K^+)$ must be supported in degrees $\ell_-=3$ and $\ell_- + \ell_+ = 7$.  It follows that $G_1 = G_2 = Sp(2)$.  As any non-trivial homomorphism from $Sp(2)$ to itself is an isomorphism, $K^+= Sp(2)$, from which it follows that $H = Sp(1)^2$.  Lastly, since $K^-/H\cong S^3$, $K^-$ must be isomorphic to $Sp(1)^3$.

\end{proof}

We now continue to classify the possible groups appearing in a cohomogeneity one action on a rational sphere as in Theorem \ref{thm:paritycase}.

\begin{proposition}\label{prop:l1notg1} We cannot have $L_1 = G_1$.

\end{proposition}

\begin{proof}
If $L_1 = G_1$, then $L' = 0$.  Since $K^+ = L_1\cdot L'$ has corank $2$ in $G$, it now follows that $G_2$ has rank $2$.  Because $L' = 0$, Proposition \ref{prop:diag} implies $K'\neq 0$.

Now, since $L' =0$, it follows that both $K^\pm \subseteq G_1\times p_2(K^-)$.  This contradicts Theorem \ref{thm:primitive}, unless $p_2(K^-) = K_2\cdot p_2(K') = G_2$.  Since $K'\neq 0$, it follows that $p_2(K') = G_2$.  Since $K'$ has rank at most $1$, we have a contradiction.

\end{proof}

%\begin{proposition}The group $p_2(K')\subseteq G_2$ must be a proper subgroup.

%\end{proposition}

%\begin{proof}Assume this is false.  Since $K'$ has rank $1$ and $G_2$ is simple, we conclude that $K'\cong G_2\cong SU(2)$.  If $H' =0$, the $L' =0$ as well since $H$ is full rank in $K^+$.  It would then follow that $G/H\cong G_1/H_1\times S^3$ and $G/K^+ = G/L_1\times S^3$.  From Proposition \ref{}, it then follows that $3 = \ell_- + \ell_+$, so $(\ell_-,\ell_+) = (1,2)$, contradicting the fact that $\ell_- \geq 3$.  Thus, $H'$ and $L'$ are non-zero.  Since $p_2:H'\rightarrw p_2(H')$ has at most finite kernel, $\rk H' = \rk L' = 1$.  Since $K^+$ is semisimple (Proposition \ref{}) and the action is irreducible, $L' = SU(2)$ and $H' = S^1$.  It now easily follows that $\pi_2(K^+/H)\otimes \Q\neq 0$, so $K^+/H\cong S^{\ell_+}$ has $\ell_+ = 2$.

%Now, consider the bundle $H'\rightarrow (K_1/H_1)\cdot K'\rightarrow K^-/H\cong S^{\ell_-}$.  Since $H'\cong S^1$, the map $\pi_3( (K_1/H_1)\cdot K')\otimes \mathbb{Q}\rightarrow \pi_3(S^{\ell_-})\otimes\Q$ must be injective.  Since $K'\cong SU(2)$, it now follows that $\ell_- = 3$.  Thus, $(\ell_-,\ell_+) = (3,2)$, giving a contradiction.

%\end{proof}

\begin{proposition}
We cannot have $H' = L'$.
\end{proposition}

\begin{proof}  Assume for a contradiction that $H' = L'$.  Then both $K^\pm \subseteq G_1\times p_2(K^-)$.  This contradicts Theorem \ref{thm:primitive} unless $p_2(K^-) = G_2$, which then implies that either $p_2(K_2) = G_2$ or  $p_2(K') = G_2$.  We consider both possibilities in turn.

Assume initially that $K_2 = p_2(K_2) = G_2$, which then implies $K'=0$.  Irreducibility of the action forces $p_2(H')\subsetneq K_2$.  Now, from the bundle $H'\rightarrow (K_1/H_1)\times K_2\rightarrow K^-/H\cong S^{\ell_-}$, if $p_1(H')$ does not act transitively on $K_1/H_1$, then from \cite[Theorem 4.8]{Totaro02}, $\dim \pi_{odd}^\Q (K^-/H) \geq 2$, giving a contradiction.  Thus,  $K_1/(H_1\cdot p_1(H'))$ must be a point, so $K_1 = H_1\cdot p_1(H')$.  But then both $K^\pm\subseteq (L_1\cdot p_1(L'))\times G_2$.  Since the action is primitive and $L_1\subsetneq G_1$ (from Proposition \ref{prop:l1notg1}), we conclude that $p_1(H') =p_1(L')= G_1$, contradicting irreducibility of the action.

So, we may assume $p_2(K') = G_2$ which then implies $K_2 = 0$.  Since $K'$ has rank $1$ and $G_2$ is simple, it follows that $K'\cong G_2=SU(2)$.   Now, since $K_2=0,$ we must have $H'\subseteq K'$, and then irreducibility implies that $H' = S^1$ or $H' = 0$.  However, $H' = S^1$ cannot occur, for otherwise, $S^{\ell_-} = K^-/H = K_1/H_1 \times K'/H' = K_1/H_1\times S^2$, giving a contradiction.

Thus, $H' = L' = 0$.  Since $K^-/H = S^{\ell_-}$, we have $\ell_ - = 3$.  Thus, we must have $\ell_+\geq 4$.  Now, since $G/K^+ = G_1/L_1\times G_2$, we conclude that $G_1/L_1$ is a rational $S^{\ell_+ + 3}$.  Since $G_1$ is simple, it follows that $L_1$ is simple.  Further, since $K^+/H = L_1/H_1 = S^{\ell_+}$, we now deduce that either $L_1 = \mathbf{G}_2$ with $H_1 = SU(3)$ or $L_1 = Spin(\ell_+ + 1)$ with $H_1 = Spin(\ell_+)$.  In the first case, from, e.g., \cite[Table B]{KapovitchZiller04}, we see that $\mathbf{G}_2$ is not the isotropy group of any transitive action on a rational sphere, so this cannot occur.

Since $\ell_+\geq 4$, the only rational spheres of the form $G_1/Spin(\ell_+ + 1)$ are standard spheres, or $Spin(\ell_+ + 3)/Spin(\ell_+ + 1)$.  However, by inspecting each one, one easily sees that $\dim G_1/Spin(\ell_+  +1)\neq \ell_+ + 3$, so we have a final contradiction.

%Since $(K_1\cdot K')/(H_1\cdot H')= S^{\ell_-}$, it follows from \cite{Totaro}[] that $p_1(H_1\cdot H') = K_1$.  Since $H'= L'$, we now see that both $K^\pm\subseteq L_1\cdot p_1(H')\times G_2$.  Primitivity forces $L_1\cdot p_1(H') = G_1$, which implies $L_1 = G_1$ or $p_1(H') = G_1$.  However, the first option contradicts Proposition \ref{}, while the second contradicts irreducibility of the action.
\end{proof}

Since $H$ has full rank in $K^+$, it follows that $H'\subsetneq L'$ is full rank and that $H_1 = L_1$.  Further, if $H_1$ is non-trivial, then Proposition \ref{prop:finiteintersection} implies that $H_1$ must act non-trivially on $K^-/H\cong S^{\ell_-}$.  In particular, $H_1$ cannot be a normal subgroup of $K_1$, unless $H_1 = 0$.

\begin{proposition}  The group $K_1$ must be non-trivial.

\end{proposition}

\begin{proof}Assume for a contradiction that $K_1$ is trivial.  Since $L_1=H_1\subseteq K_1$, both $H_1$ and $L_1$ are trivial as well, and thus, by Proposition \ref{prop:h1trivial} $G = Sp(2)\times Sp(2)$.  Since $K'$ has rank at most one, we now see that $K^-$ has rank at most two, contradicting the fact that $K^-$ has corank $1$ in $G$.

%Since $K^-$ has corank $1$ in $G$, the fact that $K_1$ is trivial implies $G_1 = SU(2)$.  Further, since $\dim \pi_{odd}(G/H)\otimes \Q = 3$, $G$ must have rank at least $3$, so $G_2$ must have rank $2$.  Since $H = H'$ has corank $2$ and the action is irreducible, we must have $H' = S^1$ and $L' = SU(2)$.  Thus, $\ell_+ = \dim(L'/H') = 2$.  Further, since $K^+\cong G_1$, it follows that $G/K^+\cong G_2$.  Since $G/K^+$ has the rational cohomology ring of $S^{\ell_-}\times S^{\ell_-  + 2}$, from the classification of rank 2 simple groups, it follows that $G_2 = SU(3)$.

%But then $\pi_\ast(G/H)\otimes \Q\cong \pi_\ast(S^2\times S^3\times S^5)\otimes \Q$.  In particular, $\ell_- = 3$, contradicting the fact that $(\ell_-,\ell_+)\neq (3,2)$.

\end{proof}

\begin{proposition} \label{prop:secondprojection} If $H_1$ is non-trivial, then we must have $p_2(H') = p_2(K^-)$.  If $H_1$ is trivial 0, then both $p_i(H')=p_i(K^-)$.

\end{proposition}

\begin{proof}Since $H\subseteq K^-$, it follows that $p_2(H') = p_2(H)\subseteq p_2(K^-) = K_2\cdot p_2(K')$.  We must show the reverse inclusion.

When $H_1$ is trivial, Proposition \ref{prop:h1trivial} gives that $H = H' \cong Sp(1)^2\subseteq Sp(2)\times Sp(2)$, while $K^-\cong Sp(1)^3$.   Note that for both $i=1,2$, $p_i(H')\cong Sp(1)^2\subseteq Sp(2)$ is maximal among connected proper subgroups.  Since $p_i(K^-)$ cannot be surjective and $p_i(H')\subseteq p_i(K^-)$, we must have $p_i(H')=p_i(K^-)$.

Thus, we may assume $H_1$ is non-trivial.  We first claim that $K_1$ is not a subset of $H_1\cdot p_1(H')$  To see this, assume it is false.  Since $H_1\subseteq K_1$ and $H_1$ is obviously normal in $H_1\cdot p_1(H')$, if $K_1\subseteq  H_1\cdot p_1(H')$, then $H_1$ is also normal in $K_1$.  As mentioned above, this cannot happen without contradicting Proposition \ref{prop:finiteintersection}. 

Up to cover, $S^{\ell_-} = K^-/H$ has the form $((K_1/H_1)\times K_2\times K')/H'$.  The previous paragraph establishes that the projection of the $H'$ action onto $K_1/H_1$ is not surjective.  Since $\dim\pi_{odd}^\Q(K^-/H) = 1$, it now follows from \cite[Theorem 4.8]{Totaro02} that the projection of the $H'$ action to $K_2\times K'$ is surjective.  Thus, $p_2(K^-)$ is contained in $p_2(H')$.

%Pulling $K^-$ back a cover of the form $K_1\times K_2\times K'$, $H$ pulls back to some subgroup $\overline{H}\subseteq K_1\times K_2\times K_'$.  Now, $K^-$ has a cover of the form $\overline{K_1}\times \overline{K_2}\times \overline{K}'$, with the overline denoting some finite cover.   The previous paragraph establishes that the projection of $\overline{H}$ to $K_1$ is not surjective.  From \cite[Theorem 4.8]{Totaro02}, because $\dim \pi_{odd}^\Q(K^-/H) = 1$, we it must be the case that    Then $K^-/H\cong \overline{K_1}\times \overline{K_2}\times \overline{K}'/\overline{H}$ for some cover $\overline{H}\rightarrow H$.  Since the projection of $H$ onto $K_1$ is not surjective, the projection of $\overline{H}$ onto $\overline{K}_1$ is also not surjective.  Now, since $\dim\pi_{odd}(S^{\ell_-})\otimes\Q = 1$, from \cite{Totaro}[], it follows that $\overline{H}$ must projection surjectively onto all but one simple factor $\overline{K_1}\times \overline{K_2}\times \overline{K}'$.  In particular, $\overline{H}$ must project onto $\overline{K_2}\times \overline{K}'$.  It now follows easily that $p_2(H') = K_2\cdot p_2(K')$.
\end{proof}

We have two immediate corollaries.

\begin{corollary}  The map $p_2:L'\rightarrow G_2$ is an isomorphism, and $G/K^+\cong G_1/L_1$.  Further $L_1$ must be simple.

\end{corollary}

\begin{proof}  If $H_1$ is trivial this is clear from Proposition \ref{prop:h1trivial}.  So, assume $H_1$ is non-trivial.  Proposition \ref{prop:secondprojection} gives $p_2(H') = p_2(K^-)$, so we see $K^\pm\subseteq G_1\times p_2(L')$.  Thus, the action is not primitive, unless $p_2(L') = G_2$.  However, the kernel of $p_2:L'\rightarrow G_2$ is finite, so this map must be a covering.  As $G_2$is simply connected, it must be an isomorphism.

Since $p_2:K^+\rightarrow G_2$ is surjective with kernel $L_1$, it follows from \cite[Lemma 1.3]{KapovitchZiller04} that $G/K^+$ is diffeomorphic to $G_1/L_1$.

The last claim follows easily since $G_1$ is simple and $\dim \pi_{even}(G_1/L_1)\otimes \Q = 0$.

\end{proof}

\begin{corollary}  If $H_1$ is non-trivial, then $K_1/H_1\cong K^-/H\cong S^{\ell_-}$.

\end{corollary}

\begin{proof}  The projection map $p_2:K^-\rightarrow p_2(K^-)\subseteq G_2$ descends to a fiber bundle map $K^-/H\rightarrow p_2(K^-)/p_2(H)$ with fiber $K_1\cdot H/H\cong K_1/H_1$.  From Proposition \ref{prop:secondprojection}, when $H_1$ is non-trivial, $p_2(K^-)/p_2(H)$ is a point, so the inclusion $K_1/H_1\cong K^-/H = S^{\ell_-}$.

\end{proof}

Summarizing the previous propositions, we have obtained the following structure result.

\begin{proposition}  Suppose $G_1$ and $G_2$ are simple compact Lie groups and $G = G_1\times G_2$ acts via a cohomogeneity one action on a rational sphere $M^n$.  Assume the homotopy fiber of the inclusion of a principal orbit into $M$ belongs to Case (4) of Theorem \ref{thm:GH} and that $\ell_-$ and $\ell_+$ have different parities.  Then

\begin{itemize}\item  $\rk G_2\leq 2$

\item  $K^+ = L_1\cdot L'$ with $L'\cong G_2$

\item  $H\cong L_1\cdot  H'$ with $L'/H'\cong S^{\ell_+}$

\item  $K^- = (K_1\times K_2)\cdot K'$ with $p_2(H') = p_2(K^-)$, and $K_1/H_1\cong S^{\ell_-}$.

\end{itemize}

\end{proposition}

 Further, since $G/K^+\cong G_1/L_1$ and $H^\ast(G/K^+;\Q)\cong H^\ast(S^{\ell_-}\times S^{\ell_- + \ell_+};\Q)$, it follows that $L_1$ is simple.  Thus, $(G_1,L_1)$ must be in Table \ref{table:corank2}.  However, most of these entries are not valid.
 
 \begin{proposition}  If $(G_1,L_1)$ is on Table \ref{table:corank2} and corresponds to an action on a rational sphere as in Theorem \ref{thm:paritycase}, then $(G_1,L_1) = (SU(n),SU(n-2))$ or $(G_1,L_1) = (Sp(n),Sp(n-2))$.
 
 \end{proposition}
 
 \begin{proof}  The following three observations can be used to rule out all other possibilities.

\begin{itemize}\item Since $L'/H' \cong S^{\ell_+}$ and $\rk L'\leq 2$, it follows that $\ell_+\in \{2,4,6\}$.

\item Since $K_1/H_1\cong S^{\ell_-}$, $H_1 = L_1$ must be the isotropy group of a transitive action on $S^{\ell_-}$ by a subgroup of $G_1$.  

\item  Since $p_1:K^+\rightarrow G_1$ is has finite kernel, the embedding of $L_1$ into $G$ must extend to an embedding of $L_1\cdot p_1(L')$ with $L' = SU(2), Sp(2), \mathbf{G}_2$ respectively, for $\ell_+ = 2,4,6$ respectively.

\end{itemize}

We illustrate the argument only for the infinite families appearing in the table.

When $(G_1,K_1) = (Spin(2n+1), Spin(2n-3))$, as $\ell_+ = 4$, $Spin(2n+1)$ must contain a copy of $Spin(2n-3)\cdot Sp(2)$.  However, elementary representation theory shows that the smallest orthogonal representation of $Spin(2n-3)\times Sp(2)$ is of dimension $2n-3 + 5 > 2n  + 1$.

When $(G_1,K_1) = (Spin(2n), Spin(2n-3)$, we have $\ell_+ = 2n-4 \in \{2,4,6\}$.  As $n\geq 4$ in the table, we must have $n \in \{4,5\}$.  When $n=4$, elementary representation theory shows that  $Spin(8)$ does not contain (up to cover) $Spin(5)\cdot Sp(2)$, and when $n=5$, $Spin(10)$ does not contain $Spin(7)\cdot \mathbf{G}_2$.

\end{proof}

To complete the proof of Theorem \ref{thm:paritycase}, for each of the pairs $$(G_1,L_1) = (SU(n), SU(n-2))\text{ and }(Sp(n), Sp(n-2)),$$ we classify the groups $H,K^\pm , G$ and embeddings.

\begin{proposition}  When $(G_1,K_1) = (SU(n), SU(n-2))$, we have $G = SU(n)\times SU(2)$, $K^+ = L_1\cdot L' = SU(n-2)\cdot SU(2)$, $H = SU(n-2)\cdot S^1$, and $K^- = (K_1\times K_2)\cdot K' = (SU(n-1)\times \{e\})\cdot S^1$.  When $n=3$, there are, up to equivalence, precisely three group diagrams corresponding to actions on rational spheres.  When $n\geq 4$, there is, up to equivalence, a unique group diagram created from these groups, corresponding to the linear action of $SU(n)\times SU(2)$ on the unit sphere $S^{4n-1}\subseteq \mathbb{C}^n\otimes_{\mathbb{C}} \mathbb{C}^2$.

\end{proposition}

\begin{proof}From Table \ref{table:corank2}, when $(G_1,K_1) = (SU(n), SU(n-2))$, we must have $\ell_+ = 2$ and $\ell_- = 2n-3$.  Thus, $G_2\cong L'\cong SU(2)$ and $H'\cong S^1$.  Since $K_1/H_1\cong S^{\ell_-} = S^{\ell_-} = S^{2n-3}$, we must have $K_1\cong SU(n-1)$.

When $n=3$, $\ell_- = 3$, so we are in the case $(\ell_-,\ell_+) = (3,2)$, so the result follows in this case from Theorem \ref{thm:23}.  Hence, we assume $n\geq 4$, there is precisely one subgroup of $G_1 = SU(n)$ which is locally isomorphic to $L_1\cdot p_1(L')\cong SU(n-2)\cdot SU(2)$ and which extends the the usual $SU(n-2)\subseteq SU(n)$.

This situation is encountered in \cite[7.13]{Asoh81}, where he shows the conjugacy class of the embeddings of $H,K^\pm$ determine the group diagram up to equivalence.

\end{proof}

\begin{proposition}\label{prop:spclass}When $(G_1,L_1) = (Sp(n),Sp(n-2))$, we have $G = Sp(n)\times Sp(2)$, $K^+ = L_1\cdot L' = Sp(n-2)\cdot Sp(2)$, $H = Sp(n-2)\cdot Sp(1)^2$, and $K^- = (K_1\times K_2)\cdot K' = (Sp(n-2)\times Sp(1))\cdot Sp(1)$.  For any $n\geq 2$, there is, up to equivalence, a unique group diagram created from these groups, corresponding to the linear action of $Sp(n)\times Sp(2)$ on the unit sphere $S^{4n-1}\subseteq \mathbb{H}^n\otimes_{\mathbb{H}} \mathbb{H}^2$.

\end{proposition}

\begin{remark}As the quaternions are non-commutative, the notation $\mathbb{H}^n\otimes_{\mathbb{H}} \mathbb{H}^2$ is to be interpreted as follows.  The skew-vector space $\mathbb{H}^n$ has the skew-field of scalars  $\mathbb{H}$ multiplying on the right, while $\mathbb{H}^2$ has $\mathbb{H}$ multiplying on the left.  Then $\mathbb{H}^n\otimes_{\mathbb{H}}\mathbb{H}^2$ is an $\mathbb{R}$-vector space of real dimension $8n$.

\end{remark}

\begin{proof}When $(G_1,L_1) = (Sp(n),Sp(n-2))$, then $\ell_+ = 4$ and $\ell_- = 4n-5$.  Thus, $G_2\cong L'\cong Sp(2)$ and $H'\cong Sp(1)^2$.  Since $K_1/H_1\cong S^{\ell_-}$ and $H_1=L_1\cong Sp(n-2)$, we must have $K_1\cong Sp(n-1)$.  Now $Sp(1)^2 \cong p_2(H')  = p_2(K^-) = K_2\cdot p_2(K')$.  Since $K_1\subseteq G_1$ has corank $1$, $p_1:K^-\rightarrow G_1$ must have non-trivial kernel, so $K_2\neq 0$.  Since $p_1(H)$ has full rank and $p_1(H)\subseteq p_1(K^-)$, we deduce that $K'\neq 0$ as well.  It follows that $K_2\cong K'\cong Sp(1)$.  It is easy to see that the embeddings of each of $H,K^\pm$ are determined up to conjugacy by these calculations.

This setup is encountered in \cite[7.15]{Asoh81}, where he shows the conjugacy classes of the embeddings of $H,K^\pm$ determine the cohomogeneity one manifold up to equivalence.

\end{proof}

\appendix
\section{Homology of Brieskorn varieties}\label{sec:appendix}

In this appendix, we compute the homology groups of the Brieskorn varieties $B_d^{2m-1}$ appearing in Theorem \ref{thm:main}.  In particular, we show

\begin{proposition}\label{prop:britop}  For any $m\geq 3$ and $d\geq 1$, $B_d^{2m-1}$ is $(m-2)$-connected.  In addition, if $m$ is odd then $H_\ast(B_d^{2m-1})\cong H_\ast(S^{m-1}\times S^m)$ for $d$ even and $H_\ast(B_d^{2m-1})\cong H_\ast (S^{2m-1})$ for $d$ odd.  If $m$ is even, then $B_d$ is a rational sphere whose only torsion is $H_{m-1}(B_d)\cong \mathbb{Z}/d\mathbb{Z}$.

\end{proposition}

Recall that $B_d^{2m-1}$ is the intersection of the zero set of the complex polynomial $z_0^d + \sum_{i=1}^m z_i^2$ with a sphere about the origin in $\mathbb{C}^{m+1}$; we always assume that $d\geq 1$ and $m\geq 3$.  We first show that $H_{m-1}(B_d)$ is cyclic, and then compute the order.

\begin{proposition}  The space $B_d^{2m-1}$ is $(m-2)$-connected with $H_{m-1}(B_d)$ cyclic.

\end{proposition}

\begin{proof} From \cite{Brieskorn66}, $B_d$ is $(m-2)$-connected, so must only show $H_{m-1}(B_d)$ is cylic.

As shown in \cite{GroveVerdianiWilkingZiller06}, the group $G = SO(2)\times SO(m)$  acts via a cohomogeneity one action with isotropy group $$K^- \cong SO(2)SO(m-2) = (e^{-i\theta}, \diag(R(d\theta), A))$$ with $R(\theta)$ being a standard rotation and $A\in SO(m-2)$.  Thus, $G/K^-\cong T^1 S^{m-1}$.  Now, $G/H$ is a circle bundle over $G/K^-$ and the Euler class must vanish because $H^2(G/K^-) = 0$.  It follows that the induced map $H_{m-1}(G/H)\rightarrow H_{m-1}(G/K^-)$ is surjective and that $H_{m-2}(G/H) = 0$.  From the Mayer-Vietoris sequence for $B_d$ associated to the double disk bundle decomposition, it now follows that $H_{m-1}(B_d)$ is a quotient of $H_{m-1}(G/K^+)$, so we need only demonstrate this group is cyclic.   Again from \cite{GroveVerdianiWilkingZiller06}, $G/K^+$ is either diffeomorphic to $S^1\times S^{m-1}$ or to a $\mathbb{Z}_2$ quotient of it with $\mathbb{Z}_2$ acting via the antipodal map on both factors.  In the latter case, viewing this space as the total space of an $S^{m-1}$-bundle over $S^1$, the Wang sequence easily shows that $H_{m-1}(G/K^+)\cong \mathbb{Z}$ in all cases.

\end{proof}

Because $B_d$ is $(m-2)$-connected, Poincar\'e duality and the universal coefficient theorem completely determine $H_\ast(B_d)$ once we have determined the order of $H_{m-1}(B_d)$.

\begin{proposition}   If $m$ is odd and $d$ is even, then $H_{m-1}(B_d)$ is infinite.  If $m$ is odd and $d$ is odd, then $H_{m-1}(B_d)$ is trivial.  If $m$ is even, then $H_{m-1}(B_d)$ has order $d$.

\end{proposition} 

\begin{proof}Via Poincar\'e and Alexander duality, we see that $H_{m-1}(B_d)\cong H^m(B_d)\cong H_m(S^{2m+1}\setminus B_d)$.  From \cite[Theorem 4.8, Theorem 6.5]{MilnorSingular}, the space $B':= S^{2m+1}\setminus B_d$ is a fiber bundle over $S^1$ with fiber $F$ having the homotopy type of a bouquet of spheres $\bigvee_{i=1}^\mu S^m$.  Note that $\mu \geq 1$ \cite[Theorem 9.1]{MilnorSingular}.  Applying the Wang sequence to this bundle, we deduce that $H_m(B')$ is isomorphic to the cokernel of the map $h_\ast - Id_{H_m(F)}:H_m(F)\cong \Z^\mu \rightarrow H_m(F)$ where $h:F\rightarrow F$ is the monodromy of this bundle.  Since $H_{m-1}(B_d)$ is cyclic, it now follows that $H_{m-1}(B_d)$ is cylic of order $|\det(h_\ast -Id_{H_m(F)})|$, unless this determinant is $0$ in which case $H_{m-1}(B_d)$ is infinite cyclic.

Set $\Delta(t) = \det(tId_{H_m(f)} - h_\ast)$, and observe that we are interested in computing $\Delta(1)$.  According to \cite[Theorem 9.1]{MilnorSingular}, $$\Delta(t) = \prod_{\omega\in\C, \omega^d = 1, \omega\neq 1} (t - \omega(-1)^m).$$

If $m$ is even, we find $\Delta(t) \cdot \frac{t-1}{t-1} = \frac{t^d - 1}{t-1}$ since the numerators are both monic degree $d$ polynomials whose roots are all of the $d$-th roots of unity.  It now easily follows that $\Delta(1) = d$ in this case.

If $m$ is odd, we find $$\Delta(-t) = \prod_{\omega^d = 1, \omega\neq 1} ( -t + \omega) = (-1)^{d-1} \left(\prod_{\omega^d = 1, \omega\neq 1} (t-\omega)\right) \cdot \frac{t-1}{t-1} = (-1)^{d-1}  \frac{t^d - 1}{t-1}.$$  Thus, $\Delta(t) = (-1)^{d-1} \frac{(-1)^d t^d - 1}{-t-1}$.  Then $\Delta(1) = 0$ if $d$ is even and $\Delta(1) = 1$ if $d$ is odd.
\end{proof}

%%%%% Bibliography %%%%%
\bibliographystyle{alpha}
\bibliography{myrefs}

\end{document}